\documentclass[14pt]{amsart}
\usepackage{amscd}
\usepackage{amsmath}
\usepackage{amsfonts}
\usepackage{bm}
\usepackage{amsfonts,amsmath,amssymb,amscd,bbm,amsthm,mathrsfs,dsfont}
\usepackage{mathrsfs}
\usepackage{pb-diagram}
\usepackage{color}
\usepackage{amssymb}
\usepackage{xypic}
\usepackage[dvips]{graphicx}

\usepackage[all]{xy}
 \usepackage{CJK}

\newtheorem{Thm}{Theorem}[section]
\newtheorem{Fac}[Thm]{Fact}
\newtheorem{Lem}[Thm]{Lemma}
\newtheorem{Def}[Thm]{Definition}
\newtheorem{Cor}[Thm]{Corollary}
\newtheorem{Prop}[Thm]{Proposition}

\newtheorem{Rem}[Thm]{Remark}

\title[Structure of cluster algebras]
{ On Structure of cluster algebras of geometric type, II:\\ Green's equivalences and paunched surfaces}

\author{Min Huang $\;\;\;\;\;\;$ Fang Li}
\address{Min huang
\newline Department
of Mathematics, Zhejiang University (Yuquan Campus), Hangzhou, Zhejiang
310027,  P.R.China}
\email{minhuang1989@hotmail.com}

\address{Fang Li
\newline Department
of Mathematics, Zhejiang University (Yuquan Campus), Hangzhou, Zhejiang
310027, P.R.China}
\email{fangli@zju.edu.cn}

\date{version of \today}

\newcommand{\lra}{\longrightarrow}

\newcommand{\ra}{\rightarrow}
\newcommand{\sdp}{\times\kern-.2em\vrule height1.1ex depth-.05ex}
\newcommand{\epi}{\lra \kern-.8em\ra}

 \normalbaselines
\begin{document}

\renewcommand{\thefootnote}{\alph{footnote}}
\setcounter{footnote}{-1} 
\renewcommand{\thefootnote}{\alph{footnote}}
 \setcounter{footnote}{-1} 
 \setcounter{footnote}{-1} \footnote{ \emph{The corresponding author}: Fang Li}

\begin{abstract} Following our previous work \cite{HLY}, we introduce the notions of partial seed homomorphisms and partial ideal rooted cluster morphisms.  Related to the theory of Green's equivalences, the isomorphism classes of sub-rooted cluster algebras of a rooted cluster algebra are corresponded one-by-one to the regular $\mathcal D$-classes of the semigroup consisting of partial seed endomorphisms of the initial seed. Moreover, for a rooted cluster algebra from a Riemannian surface, they are also corresponded to the isomorphism classes of the so-called paunched surfaces.
\end{abstract}

\maketitle
\bigskip

\section{Introduction}

Cluster algebras are commutative algebras that were introduced
by Fomin and Zelevinsky \cite{fz1} in order to give a
combinatorial characterization of total positivity and canonical
bases in algebraic groups. The theory of cluster
algebras is related to numerous other fields. Since its introduction, the study on cluster algebras mainly involves intersection with Lie theory, representation theory of algebras, its combinatorial method (e.g. the periodicity issue) and categorification, Riemannian surfaces and its topological setting, including the Teichm$\ddot{u}$ller theory.

The algebraic structure and properties of cluster algebras were originally studied in a series of articles \cite{fz1,fz2,[1],fz4} involving bases and the positivity conjecture.

In our previous work \cite{HLY}, we focus to characterize the internal structure of rooted cluster algebras, including all sub-rooted cluster algebras and rooted cluster quotient algebras, via sub-seeds and seed homomorphisms.

In this work, as a subsequent of \cite{HLY}, we introduce the semigroup of partial seed
endomorphisms of the initial seed for a rooted cluster algebra and  partial ideal rooted cluster endomorphisms and
 find the connection of the internal structure of a
cluster algebra with the semigroup of  partial seed endomorphisms of the
initial seed via the method of Green's equivalences in semigroup theory. We build the correspondence
between the isomorphism classes of sub-rooted cluster algebras and
the regular $\mathcal D$-classes of the semigroup of  partial seed
endomorphisms of the initial seed for a rooted cluster algebra. This result supplies a window for
one to understand the internal structure of a rooted cluster algebra
through some algebraic properties in semigroup theory.

One of the main results is Theorem \ref{number}, for which we combine the method of mixing-type sub-seeds (Definition \ref{mixingseed}) and the Green's equivalences in the algebraic theory of semigroups as an original idea.

Another main result, Theorem \ref{final}, in fact, gives a new algebraic invariant of the internal structure of a Riemannian surface. As preparation,  we introduce the so-called paunched surfaces of a Riemannian surface (Definition \ref{paunch}) which are shown to correspond to sub-rooted cluster algebras of the rooted cluster algebra from the Riemannian surface, using of the theory from \cite{HLY}, and then using Theorem \ref{number}, to correspond to regular $\mathcal D$-classes of the semigroup consisting of partial seed homomorphisms (Definition \ref{partial seed}).

The organization of this article is the following. In Section 2,  we recall the definition of cluster algebras of geometric type. In Section 3,
 some required concepts and results are given from \cite{HLY}\cite{ADS}, and then introduce the semigroups of partial seed homomorphisms of the initial seed for a rooted cluster algebra and partial ideal rooted cluster endomorphisms (Definition \ref{partial}).

In Section 4, we introduce the theory of Green's equivalences in semigroup theory and then give some preliminaries on the semigroup of partial seed homomorphisms.

In Section 5, we give the proofs of the main conclusions, Theorem \ref{greenrelations}, Theorem \ref{greens} and Theorem \ref{number}, including three parts, i.e. the characterization of Green's equivalences in $\rm End_{par}(\Sigma)$, the continuation of Green's equivalences from $\rm End_{par}(\Sigma)$ to $\rm End_{par}(\mathcal A(\Sigma))$ and characterizing iso-classes of sub-rooted cluster algebras via regular $\mathcal D$-classes.

In Section 6, we firstly characterize sub-rooted cluster algebras of rooted cluster algebras from Riemannian
surfaces via paunched surfaces. Then, the proof of the main result, Theorem \ref{final}, is given.

Here we give the four main conclusions in this paper.

\begin{Thm}\label{greenrelations}
Let $f:\Sigma_{I_0,I_1}\rightarrow\Sigma$ and
$f':\Sigma_{I'_0,I'_1}\rightarrow\Sigma$ be two regular partial seed
homomorphisms in $\rm{End}_{par}(\Sigma)$, then

(1)~$f\mathcal{R}f'$ if and only if ~(a)~
$f(\Sigma_{I_0,I_1})=f'(\Sigma_{I'_0,I'_1})$.

(2)~ $f\mathcal{L}f'$ if and only if ~(b)~
$\Sigma_{I_0,I_1}=\Sigma_{I'_0,I'_1}$, and ~(c)~ there exists an
isomorphism $g: f(\Sigma_{I_0,I_1})\rightarrow
f'(\Sigma_{I'_0,I'_1})$ such that $f'=(id_{I''_0,I''_1}g)\circ f$, where $f'(\Sigma_{I'_0,I'_1})=\Sigma_{I''_0,I''_1}$.

(3)~ $f\mathcal{H}f'$ if and only if the above statements (a), (b) and (c) hold.

(4)~ $f\mathcal{D}f'$ if and only if $f(\Sigma_{I_0,I_1})\cong
f'(\Sigma_{I'_0,I'_1})$.
\end{Thm}

\begin{Thm}\label{greens}
Let $f:\mathcal{A}(\Sigma_{I_0,I_1})\rightarrow\mathcal{A}(\Sigma)$
and
$f':\mathcal{A}(\Sigma_{I'_0,I'_1})\rightarrow\mathcal{A}(\Sigma)$
be noncontractible ideal partial rooted cluster morphisms in
$\rm{End}_{par}(\mathcal{A}(\Sigma))$ with restricted partial seed
homomorphisms $f^S, f'^S\in\rm{End}_{par}(\Sigma)$, respectively.
Assume that $f^S$ and $f'^S$ are regular. Then, $f^S\mathcal{F}f'^S$ in
$\rm{End}_{par}(\Sigma)$ if and only if $f\mathcal{F}f'$ in
$\rm{End}_{par}(\mathcal{A}(\Sigma))$ with $\mathcal F$ one of
the Green's equivalences $\mathcal R, \mathcal L, \mathcal H,$ or $\mathcal
D$.
\end{Thm}

\begin{Thm}\label{number}  For a rooted cluster algebra $\mathcal{A}(\Sigma)$, let $\mathcal{A}(\Sigma')$ denote any of its sub-rooted cluster algebra and $[\mathcal{A}(\Sigma')]$ the iso-class of $\mathcal{A}(\Sigma')$ in {\bf Clus}. There exists a one-to-one correspondence between the iso-classes of sub-rooted cluster
algebras of $\mathcal{A}(\Sigma)$ and the regular
$\mathcal{D}$-classes in $\rm{End_{par}}(\Sigma)$ through the
bijection $\varphi: [\mathcal{A}(\Sigma')]\longrightarrow
D_{id_{I_0,I_1}}$ with $\Sigma'\cong \Sigma_{I_0,I_1}$ for some
$I_0\subseteq X$ and $I_1\subseteq \widetilde{X}$.
\end{Thm}

 \begin{Thm}\label{final}
                             For a Riemannian surface $(S,M)$ with triangulation $T$ without punctures,

(a)\; there is a one-to-one correspondence between the iso-classes of $(I_0,I_1)$-paunched surfaces of $(S,M,T,\mathcal L)$ and the iso-classes of sub-rooted cluster algebras of $\mathcal A(\Sigma(S,M,T,\mathcal L))$;

(b)\; there is a one-to-one correspondence between the iso-classes of $(I_0,I_1)$-paunched surfaces of $(S,M,T,\mathcal L)$ and the set of regular $\mathcal D$-classes of $\rm End_{par}(\Sigma(S,M,T,\mathcal L))$.
\end{Thm}
\bigskip

\section{On the definition of cluster algebras of geometric type}
The original definition of cluster algebra given in \cite{fz1} is in terms of exchange pattern. We recall the equivalent definition in terms of seed mutation in \cite{fz2}; for more details, refer to \cite{GSV, fz1, fz2}.

Let $(\mathds P,\oplus,\cdot)$ be a {\bf semi-field}, i.e., an abelian
multiplicative group, supplied with an {\bf auxiliary addition $\oplus$}, which is commutative,
associative, and distributive with respect to the multiplication in $\mathds P$. In particular, let $\mathds P$ be a free abelian multiplicative group with a finite set of generators $p_j (j\in J)$, and with an addition $\oplus$ given by \begin{equation}\label{grouprelation}
\prod\limits_{j\in J}p_j^{a_j}\oplus \prod\limits_{j\in J}p_j^{b_j}=\prod\limits_{j\in J}p_j^{min \{a_j,b_j\}}\end{equation}
 for any $a_j, b_j\in \mathds{Z}$.
 In this case, the semi-field $\mathds P$ is called as {\bf tropical semi-field}, denoted by $Trop(p_j,j\in J)$, which is related to tropical geometry (refer to \cite{fz2} and \cite{BS}).

For a semi-field $\mathds P$, it was proved in \cite{fz1} that the group ring $\mathds{ZP}$ is a domain. Let $\mathds F$ (the {\bf ambient field} of cluster algebra) be the fraction field of the polynomial $\mathds {ZP}[u_1,\cdots,u_n]$ in $n$ independent variables $\{u_1,\cdots,u_n\}$ ($n$ is the {\bf rank} of cluster algebra). In particular, when $\mathds P=Trop(p_j,j\in J)$, we have $\mathds F\cong \mathds Q(u_i, p_j)$  $(i=1,\cdots, n, j\in J)$.

An $n\times n$ integer matrix $A=(a_{ij})$ is called {\bf
sign-skew-symmetric} if either $a_{ij}=a_{ji}=0$ or $a_{ij}a_{ji}<0$ for any $1\leq i,j\leq
n$.

An $n\times n$ integer matrix $A=(a_{ij})$ is called {\bf
skew-symmetric} if $a_{ij}=-a_{ji}$ for all $1\leq i,j\leq n$.

An $n\times n$ integer matrix $A=(a_{ij})$ is called {\bf
$D$-skew-symmetrizable} if $d_ia_{ij}=-d_ja_{ji}$ for all $1\leq
i,j\leq n$, where $D$=diag$(d_i)$ is a diagonal matrix with all
$d_i\in \mathbb{Z}_{\geq 1}$.

Let $\widetilde{A}$ be an $n\times (n+m)$ integer matrix whose
principal part, denoted as $A$, is the $n\times n$ submatrix formed
by the first $n$ rows and the first $n$ columns. The entries of
$\widetilde A$ are written by $a_{xy}$, $x\in X$ and $y\in
\widetilde{X}$. We say $\widetilde A$ to be {\bf
sign-skew-symmetric} (respectively, {\bf skew-symmetric}, $D$-{\bf
skew-symmetrizable}) whenever $A$ possesses this property.

For two $n\times (n+m)$ integer
 matrices $A=(a_{ij})$ and
$A'=(a'_{ij})$, we say that $A'$ is
obtained from $A$ by a {\bf matrix mutation $\mu_i$} in direction $i, 1\leq i\leq n$, represented as $A'=\mu_i(A)$, if
\begin{equation}\label{matrixmutation}
a'_{jk}=\left\{\begin{array}{lll} -a_{jk},& \text{if}
~~j=i \ \text{or}\;
k=i;\\a_{jk}+\frac{|a_{ji}|a_{ik}+a_{ji}|a_{ik}|}{2},&
\text{otherwise}.\end{array}\right.\end{equation}

It is easy to verify that $\mu_i(\mu_i(A))=A$. The
skew-symmetric/symmetrizable property of matrices is invariant under
mutations. However, the sign-skew-symmetric property is not so. For
this reason, a sign-skew-symmetric matrix $A$ is called {\bf totally
sign-skew-symmetric} if any matrix, that is mutation equivalent to
$A$, is sign-skew-symmetric.

A {\bf seed} in $\mathds F$ is a triple $\Sigma=(X, P, B)$, where\\
(a)~$X=\{x_1,\cdots x_n\}$ is a transcendence basis of $\mathds F$ over the fraction field of $\mathds{ZP}$, which is called a {\bf cluster}, whose each $x\in X$ is called a {\bf cluster variable} (see \cite{fz2});\\
(b)~$P=(p^{\pm }_x)_{x\in X}$ is a $2n$-tuple of elements of $\mathds P$ satisfying the {\bf normalization condition} $p^{+}_x\oplus p^{-}_x=1$ for all $x\in X$, which is called the {\bf coefficient tuple} (see \cite{fz4});\\
(c)~$B=(b_{xy})_{x,y\in X}$ is an $n\times n$ integer matrix with rows and columns indexed
by $X$, which is totally sign-skew-symmetric.

Let $\Sigma=(X,P,B)$ be a seed in $\mathds F$ with $x\in X$, the {\bf mutation} $\mu_x$ of $\Sigma$ at $x$ is defined satisfying $\mu_x(\Sigma)=(\mu_x(X), \bar P, \mu_{x}(B))$ such that

(a)~ The {\bf adjacent cluster} $\mu_{x}(X)=\{\mu_{x}(y)\mid y\in X\} $, where  $\mu_{x}(y)$ is given by the {\bf exchange relation}
\begin{equation}
\mu_{x}(y)=\left\{\begin{array}{lll} \frac{p^{+}_x\prod\limits_{t\in X, b_{xt}>0}t^{b_{xt}}+p^{-}_x\prod\limits_{t\in X, b_{xt}<0}t^{-b_{xt}}}{x},& \text{if}
~~y=x;\\y,& \text{if}~~ y\neq x. \end{array}\right.\end{equation}
This new variable $\mu_{x}(x)$ is also called as {\em new} {\bf cluster
variable}.

(b)~  the $2n$-tuple $\bar P=\{\bar p^{\pm}_y\}_{y\in \mu_x(X)}$ is uniquely determined by the normalization conditions $\bar p^{+}_{y}\oplus \bar p^{-}_{y}=1$ together with
\begin{equation}
\bar p^{+}_y/\bar p^{-}_y=\left\{\begin{array}{lll}  p^{-}_x/ p^{+}_x ,& \text{if}
~~y=\mu_x(x); \\(p^{+}_x)^{b_{xy}}p^{+}_y/ p^{-}_y,&\text{if}~~ b_{xy}\geq 0;
 \\(p^{-}_x)^{b_{xy}}p^{+}_y/ p^{-}_y,&\text{if}~~ b_{xy}\leq 0. \end{array}\right.\end{equation}

(c)~ $\mu_{x}(B)$ is obtained from B by applying the matrix mutation in direction
$x$ and then relabeling one row and one column by replacing $x$ with $\mu_x(x)$.

It is easy to see that the mutation $\mu_x$ is an involution, i.e., $\mu_{\mu_x(x)}(\mu_x(\Sigma))=\Sigma$.

Two seeds $\Sigma'$ and $\Sigma''$ in $\mathds F$ are called {\bf mutation equivalent} if there exists a series of mutations $\mu_{y_1},\cdots,\mu_{y_s}$ such that $\Sigma''=\mu_{y_s}\cdots\mu_{y_1}(\Sigma')$.

 Trivially, mutation equivalence of seeds gives an equivalent relation on the set of seeds in $\mathds F$.

Let $\Sigma$ be a seed in $\mathds F$. Denote by $\mathcal S$ the set of all seeds mutation equivalent to $\Sigma$. In particular, $\Sigma\in \mathcal S$. For any $\bar \Sigma\in\mathcal S$, we have $\bar \Sigma=(\bar X,\bar P,\bar B)$. Denote $\mathcal P=\cup_{\bar\Sigma\in\mathcal S}\bar P \subseteq \mathds P$ and $\mathcal X=\cup_{\bar\Sigma\in\mathcal S}\bar X$. $\mathds Z[\mathcal P]$ represents the sub-ring of $\mathds {ZP}$ generated by $\mathcal P$.

\begin{Def}\label{clusteralgebra} \cite{fz1}\cite{fz2}
 Let $\Sigma$ be a seed in $\mathds F$.
 \\
 (i)~ The {\bf cluster algebra}, denoted by $\mathcal A=\mathcal A(\Sigma)$, associated with $\Sigma$ is defined to be the $\mathds Z[\mathcal P]$-subalgebra of $\mathds F$ generated by $\mathcal X$. $\Sigma$ is called the {\bf initial seed} of $\mathcal A$.
 \\
 (ii)~ In particular, if the semi-field $\mathds P$ is tropical, i.e. $\mathds P=Trop(x_{n+1},\cdots,x_{n+m})$ for some nonnegative integer $m$, the cluster algebra $\mathcal A$ is said to be of {\bf geometric type}.
\end{Def}

In case $\mathds P=Trop(x_{n+1},\cdots,x_{n+m})$, for any seed $\Sigma=(X, P, B)$ in $\mathds F$,
$P=(p^{\pm }_x)_{x\in X}$ is a $2n$-tuple of elements of $\mathds P$ satisfying the  normalization condition $p^{+}_x\oplus p^{-}_x=1$ for all $x\in X$. Since $\mathds P$ is a free abelian group generated by $x_{n+1},\cdots,x_{n+m}$ satisfying the relation (\ref{grouprelation}), we can denote $p_{x_i}^{+}=Y^{\textbf{\emph a}_i}$ and $p_{x_i}^{-}=Y^{\textbf{\emph b}_i}$, where $Y^{\textbf{\emph a}_i}=\prod\limits_{j=1}^{m}x_{n+j}^{a_j}$ for $\textbf{\emph a}_i=(a_{i1},\cdots,a_{im})$ and $Y^{\textbf{\emph b}_i}=\prod\limits_{j=1}^{m}x_{n+j}^{b_j}$ for $\textbf{\emph b}_i=(b_{i1},\cdots,b_{im})$ with $a_{ij},b_{ij}\in \mathds{Z}_{\geq 0}, j=1,\cdots,m$, by (\ref{grouprelation}) and the normalization condition. It is easy to see that $\textbf{\emph a}_i$ and $\textbf{\emph b}_i$ can be determined uniquely by $\textbf{\emph a}_i-\textbf{\emph b}_i$.
Precisely, if $\textbf{\emph a}_i-\textbf{\emph b}_i\overset{\Delta}{=}\textbf{\emph c}_i=(c_{ij})$, then $a_{ij}=max\{c_{ij},0\}$ and $b_{ij}=max\{-c_{ij},0\}$ for $1\leq j\leq m$.

Following the above discussion, for a given $P$ from the seed $\Sigma=(X,P,B)$, we give uniquely an $n\times m$ integer matrix $B_1=\left(\begin{array}{c}
\textbf{\emph c}_1 \\
\cdots\\
\textbf{\emph c}_n
\end{array}\right)$; conversely, for a given $n\times m$ integer matrix $B_1=\left(\begin{array}{c}
\textbf{\emph c}_1 \\
\cdots\\
\textbf{\emph c}_n
\end{array}\right)$ with $\textbf{\emph c}_i\in \mathds{Z}^m$ for any $i$, we can construct a $2n$-tuple $P=(p^{\pm}_{x_i})_{x_i\in X}$ with $p^{+}_{x_i}=Y^{\textbf{\emph a}_i}$ and $p^{+}_{x_i}=Y^{\textbf{\emph b}_i}$, where $\textbf{\emph a}_i$ and $\textbf{\emph b}_i$ are decided uniquely by $\textbf{\emph c}_i$ (see \cite{fz1}).

Hence, we can replace $P$ by $B_1$, and then denote the seed $\Sigma=(X,P,B)$ as $\Sigma=(X,\widetilde{B})$ equivalently, where $\widetilde{B}=(B\;B_1)$.

For this cluster algebra $\mathcal A(\Sigma)$ of geometric type, we can write about its seed and mutation equivalence, briefly described as follows. For more details, refer to \cite{fz1}\cite{GSV}, etc.

Give a field $\mathds F$ as an extension of the rational number field $\mathds Q$, assume that $x_1,\cdots,x_n,x_{n+1},\cdots,x_{n+m}\in \mathds F$ are $n+m$ algebraically independent over $\mathds Q$ for a non-negative integer $n$ and a non-negative integer $m$ such that $\mathds F=\mathds{Q}(x_1,\cdots,x_n,x_{n+1},\cdots,x_{n+m})$, the field of rational functions in the  set $\widetilde X=\{x_1,\cdots,x_n,x_{n+1},\cdots,x_{n+m}\}$ with coefficients in $\mathds{Q}$. We call $\widetilde{X}$ an {\bf
extended cluster}, the subset $X=\{x_1,\cdots,x_n\}$ the {\bf cluster} and the subset  $X_{fr}=\{x_{n+1},\cdots,x_{n+m}\}$  the {\bf frozen cluster} or, say, the {\bf frozen part} of $\widetilde X$ in $\mathds F$, where $x_1,\cdots,x_n$ is said to be {\bf (original) exchangeable cluster variables} and $x_{n+1},\cdots,x_{n+m}$ the {\bf stable (cluster) variables} or say {\bf frozen (cluster) variables} and $1,\cdots,n$ the {\bf exchangeable
vertices} and $n+1,\cdots,n+m$ the {\bf frozen vertices} of $\widetilde X$. In particular, when $n=0$, let $X=\emptyset$ and then $\widetilde X=X_{fr}$.

 The {\bf seed} $\Sigma$ in $\mathds F$ is a pair $\Sigma=(X,\widetilde B)$, where $\widetilde{B}=(B\;B_1)$ is a totally sign-skew-symmetric $n\times (n+m)$ matrix over $\mathds{Z}$. With this, we also denote $\mathds F=\mathds F(\Sigma)$.
The $n\times n$ matrix $B$ is called the {\bf exchange matrix} and
$\widetilde B$ the {\bf extended exchange matrix} corresponding to the seed $\Sigma$.

 In the seed $\Sigma=(X,\widetilde B)$, if $X=\emptyset$, that is,
$\widetilde{X}=X_{fr}$, we call the seed a {\bf trivial seed}.

Given the seed $\Sigma=(X,\widetilde{B})$ and $x,y\in
\widetilde{X}$, we say $(x,y)$ is a {\bf connected pair} in case (i) $x=y$ or; (ii) $x\neq y$ but $b^2_{xy}+b^2_{yx}\neq 0$ with $\{x,y\}\cap X\neq \emptyset$. A seed $\Sigma$ is defined to be {\bf connected} if for any $x,y\in \widetilde{X}$,
there exists a sequence of variables
$(x=z_0,z_1,\cdots,z_s=y)\subseteq \widetilde{X}$ such that
$(z_i,z_{i+1})$ are connected pairs for all $0\leq i\leq s-1$.

Note that in the case of geometric type, the exchange relation in direction $x\in X$ is given as:
\begin{equation}\label{exchangerelation}
\mu_{x}(y)=\left\{\begin{array}{lll} \frac{\prod\limits_{t\in \widetilde{X}, b_{xt}>0}t^{b_{xt}}+\prod\limits_{t\in \widetilde{X}, b_{xt}<0}t^{-b_{xt}}}{x},& \text{if}
~~y=x;\\y,& \text{if}~~ y\neq x. \end{array}\right.\end{equation}

Given a seed $\Sigma=(X,\widetilde B)$, another seed
$\Sigma'=(X',\widetilde B')$ is said to be {\bf adjacent} to
$\Sigma$ (in direction $x\in X$) if $X'=\mu_x(X) $  and $\widetilde B' = \mu_x(\widetilde B)$.
Moreover, we denote
$\Sigma' = \mu_{x}(\Sigma) = (\mu_{x}(X),\mu_{x}(\widetilde{B}))$.

In this paper, beginning from Section 3, cluster algebras will always be  {\em cluster algebras of geometric type} defined above.

\section{Partial seed homomorphisms and  partial ideal rooted cluster morphisms}

Seed homomorphisms are originally introduced in \cite{HLY} for studying the interne structure of (rooted) cluster algebras. Here, we recall those useful definitions and results.

 For the initial seed $\Sigma=(X,\widetilde{B})$ of a cluster algebra $\mathcal A$ and two pairs $(x,y)$ and $(z,w)$ with $x,z\in X$ and
 $y,w\in\widetilde{X}$, we say that $(x,y)$ and $(z,w)$ are {\bf adjacent
pairs} if $b_{xz}\neq 0$ or $x=z$.

\begin{Def}\label{seedhom}(\cite{HLY})
Let $\Sigma=(X,\widetilde{B})$ and $\Sigma'=(X',\widetilde{B'})$ be two
seeds. A map $f$ from $\widetilde{X}$ to $\widetilde{X'}$ is called a {\bf seed homomorphism} from the seed $\Sigma$ to the seed $\Sigma'$ if it satisfies that
 \\
 (a) ~ $f(X)\subseteq X'$ and;\\
 (b) ~ for any adjacent pairs $(x,y)$ and $(z,w)$ with $x,z\in X$ and
 $y,w\in\widetilde{X}$,
 \begin{equation}
(b'_{f(x)f(y)}b_{xy})(b'_{f(z)f(w)}b_{zw})\geq 0 \;\;\;\;\;\;\;\text{and}\;\;\;\;\;\;\; |b'_{f(x)f(y)}|\geq |b_{xy}|.
\end{equation}
\end{Def}

For seed homomorphisms $f:\Sigma\rightarrow\Sigma'$ and
$g:\Sigma'\rightarrow\Sigma''$, define their composition $gf:\Sigma\rightarrow\Sigma''$
satisfying that $gf(x)=g(f(x))$ for all $x\in \widetilde{X}$. Then we can define the {\bf seed category}, denoted as
\textbf{Seed}, to be the category whose objects are all seeds and
whose morphisms are all seed homomorphisms with composition defined as
above.

 \begin{Def}\label{mixingseed}
 Let $\Sigma=(X,\widetilde{B})$ be a seed with $\widetilde{B}$
an $n\times (n+m)$ totally sign-skew-symmetric  integer matrix. Assume $I_{0}$
is a subset of $X$, $I_{1}$ a subset of $\widetilde{X}$ with
$I_{0}\cap I_{1}=\emptyset$ and $I_{1}=I_{1}'\cup I_{1}''$ for
$I_{1}'=X\cap I_{1}$ and $I_{1}''=X_{fr}\cap I_{1}$. Denoting
$X'=X\backslash (I_{0}\cup I_{1}')$,
$\widetilde{X'}=\widetilde{X}\backslash I_{1}$ and $\widetilde{B'}$
as a $\sharp X' \times \sharp \widetilde{X'}$-matrix with
$b_{xy}'=b_{xy}$ for any $x\in X'$ and $y\in \widetilde{X'}$, one can
define the new seed $\Sigma_{I_{0},I_{1}}=(X',\widetilde{B'})$, which
is called a {\bf mixing-type sub-seed} or, say, {\bf $(I_{0},I_{1})$-type sub-seed}, of the seed
$\Sigma=(X,\widetilde{B})$.
\end{Def}

\begin{Def}\label{seediso}
Let $\Sigma$ and $\Sigma'$ be two seeds and $f:\Sigma\rightarrow\Sigma'$
be a seed homomorphism. $f$ is called a {\bf seed isomorphism} if $f$
induces bijections $X\rightarrow X'$ and $\widetilde{X}\rightarrow
\widetilde{X'}$ and $|b_{xy}|=|b'_{f(x)f(y)}|$ for all $x\in X$ and
$y\in \widetilde{X}$.
\end{Def}

Trivially, we have the following lemmas by the definitions of
seed homomorphisms and seed isomorphisms.

\begin{Lem}\label{isomorphism}
A seed homomorphism $f:\Sigma\rightarrow\Sigma'$ is a seed isomorphism if
and only if there exists a unique seed homomorphism
$f^{-1}:\Sigma'\rightarrow \Sigma$ such that $f^{-1}f=id_{\Sigma}$
and $ff^{-1}=id_{\Sigma'}$.
\end{Lem}

Similar as the definition for rooted cluster morphism, we define the image seed of a seed homomorphism.

\begin{Def} \label{imageseed}
 Let $f:\Sigma\rightarrow\Sigma'$ be a seed homomorphism. The {\bf image seed} of $\Sigma$ under $f$ is defined to be
$f(\Sigma)=(f(X),B'')$, where
$B''$ is a $\#(f(X)) \times \#(f(\widetilde{X}))$-matrix with $b''_{xy}=b'_{xy}$ for any $x\in
 f(X)$ and $y\in f(\widetilde{X})$.
\end{Def}

Recall that a seed homomorphism $f:\Sigma\rightarrow \Sigma'$ is called {\bf injective} if $f(\Sigma)$ is a mixing-type sub-seed of $\Sigma'$, see \cite{HLY}.

\begin{Def}\label{partial seed}
 Define a {\bf partial seed homomorphism} $f$ from $\Sigma$ to
$\Sigma'$ to be a seed homomorphism $f$ from an
$(I_0,I_1)$-type sub-seed $\Sigma_{I_0,I_1}$ of $\Sigma$ to $\Sigma'$,
which is denoted as $(f, \Sigma_{I_0, I_1}, \Sigma')$ or briefly as
$f$. Denote $\Sigma_{f}=\Sigma_{I_0,I_1}$ and $\rm{Hom_{par}}(\Sigma,\Sigma')$ be the set of all partial seed
homomorphisms from $\Sigma$ to $\Sigma'$.
\end{Def}

Denote
$\widetilde{X}_{\Sigma_f}=Dom(f)=\widetilde{X}\setminus I_1$ as the domain of $f$ and $(\widetilde{X}_{\Sigma_f})_{ex}=Dom(f)_{ex}=X\setminus (I_0\cup I_1)$ as the exchangeable part and $(\widetilde{X}_{\Sigma_f})_{fr}=Dom(f)_{fr}=(X_{fr}\cup I_0)\setminus I_1$ as the frozen part of $Dom(f)$ such that $Dom(f)=Dom(f)_{ex}\cup Dom(f)_{fr}$.

 For an $(I_0,I_1)$-type sub-seed
$\Sigma_{I_0,I_1}=(X\setminus I_0, \widetilde{B}')$ of a seed
$\Sigma=(X,\widetilde{B})$,  we define a special seed homomorphism
$id_{I_0,I_1}:\Sigma_{I_0,I_1}\rightarrow\Sigma$ as follows.

 $id_{I_0,I_1}|_{\widetilde{X}\setminus I_1}$ is just the natural
inclusion from $\widetilde{X}\setminus I_1$ to $\widetilde{X}$. Then,
Definition \ref{seedhom} (a) holds for $id_{I_0,I_1}$.

We have $b'_{xy}=b_{xy}$ for all $x\in X\setminus I_0$ and $y\in
\widetilde{X}\setminus I_1$ by the definition of $\Sigma_{I_0,I_1}$.
Then for any adjacent pairs $(x,y)$ and $(z,w)$ with $x,z\in
X\setminus I_0$, $y,w\in \widetilde{X}\setminus I_1$ , trivially
$$b'_{xy}b_{xy}b'_{zw}b_{zw}\geq 0\;\;\;\text{and}\;\;\;
|b'_{xy}|=|b_{xy}|.$$ But $id_{I_0,I_1}(u)=u$ for $u=x,y,z,w$.
Hence, Definition \ref{seedhom} (b) holds.

 By Definition \ref{imageseed}, we have
$(id_{I_0,I_1})(\Sigma_{I_0,I_1})=\Sigma_{I_0,I_1}$, which means
that $id_{I_0,I_1}$ is an injective seed homomorphism.

Of course, as a partial seed homomorphism, $id_{I_0,I_1}\in
\rm{End_{par}}(\Sigma)$.

For two partial seed homomorphisms $(f, \Sigma_{I_0, I_1},
\Sigma')\in \rm{Hom_{par}}(\Sigma,\Sigma')$ and $(g, \Sigma'_{J'_0,
J'_1}, \Sigma'')\in \rm{Hom_{par}}(\Sigma',\Sigma'')$, their {\bf
composition of partial
seed homomorphisms}, denoted as \\
 \centerline{$(g,
\Sigma'_{J'_0, J'_1}, \Sigma'')\circ(f, \Sigma_{I_0, I_1}, \Sigma')$\;\;
or briefly as\;\; $g\circ f$,} is defined as $(g\circ f)(x)=g(f(x))$ for all $x\in Dom(g\circ f)$ if it satisfies:
$$Dom(g\circ f)=Dom(g\circ f)_{ex}\cup Dom(g\circ f)_{fr},$$ where
 \begin{equation}\label{ex-f}
Dom(g\circ f)_{ex}=\{x\in Dom(f)_{ex}| f(x)\in Dom(g)_{ex}\},\;\;\text{and}\;\;Dom(g\circ f)_{fr}=\{x\in Dom(f)_{fr}| f(x)\in Dom(g)_{fr}\}.
 \end{equation}

 We claim that
$g\circ f$ is a partial seed homomorphism from $\Sigma$ to $\Sigma''$,
that is, the set of partial seed homomorphisms is closed under this
composition.

In fact, for $g\circ f$, the condition (a) of Definition
\ref{seedhom} is clear. For (b), $(x,y)$ and
$(z,w)$ are adjacent pairs in $\Sigma$ with $x,z\in X\cap Dom(g\circ f)$
and $y,w\in Dom(g\circ f)$. Since $f$ is partial seed homomorphism,
therefore, we have $f(x),f(z)\in X'$ and
\begin{equation}\label{equ:*}
 (b'_{f(x)f(y)}b_{xy})(b'_{f(z)f(w)}b_{zw})\geq 0
\;\;\;\text{and}\;\;\; |b'_{f(x)f(y)}|\geq|b_{xy}|.
\end{equation}
 Thus, $(f(x),f(y))$ and $(f(z),f(w))$ are adjacent pairs in $\Sigma'$
with $f(x),f(z)\in X'\cap Dom(g)$ and $f(y),f(w)\in Dom(g)$. As $g$
is a partial seed homomorphism, we have
\begin{equation}\label{equ:**}
(b''_{g\circ f(x)g\circ f(y)}b'_{f(x)f(y)})(b''_{g\circ f(z)g\circ f(w)}b'_{f(z)f(w)})\geq 0
\;\;\;\text{and} \;\;\;|b''_{g\circ f(x)g\circ f(y)}|\geq|b'_{f(x)f(y)}|.
\end{equation}
By (\ref{equ:*}) and (\ref{equ:**}), we have
\begin{equation}
(b''_{g\circ f(x)g\circ f(y)}b_{xy})(b''_{g\circ f(z)g\circ f(w)}b_{zw})\geq 0
\;\;\;\text{and} \;\;\;|b''_{g\circ f(x)g\circ f(y)}|\geq|b_{xy}|.
\end{equation}
Therefore, $g\circ f$ satisfies the condition (b) of Definition
\ref{seedhom}.

Hence, for $\Sigma=\Sigma'$, $\rm{End_{par}}(\Sigma)$ is closed under this
composition, which is considered as the multiplication. Below, we will prove that the associative law holds for this multiplication in $\rm{End_{par}}(\Sigma)$.

\begin{Rem}
In Definition \ref{partial seed}, we allow the mixing-type sub-seed
$\Sigma_{I_0, I_1}$ to be the empty seed $\emptyset$.
  In this case, we say $f$ to be the {\bf empty homomorphism} too, denoted as $\emptyset$. Thus, if $f^{-1}(\widetilde{X'}\setminus I'_1)\subseteq I_1$, then $g\circ f=\emptyset$.
\end{Rem}

\begin{Fac}\label{dom}
If $f\in \rm End_{par}(\Sigma,\Sigma')$ and $g\in \rm End_{par}(\Sigma',\Sigma'')$, then

(a)~ $\Sigma_{g\circ f}$ is a  mixing-type sub-seed of $\Sigma_{f}$. In general, we can denote $\Sigma_{g\circ f}=(\Sigma_{f})_{\emptyset, I_1}$, where $I_1=\{x\in Dom(f)_{ex}\;|\;f(x)\not\in Dom(g)_{ex}\}\cup \{x\in Dom(f)_{f}\;|\;f(x)\not\in Dom(g)_{f}\}$ and

(b)~ The image seed $(g\circ f)(\Sigma_{g\circ f})$ is a mixing-type sub-seed of $g(\Sigma_{g})$.
\end{Fac}

\begin{Rem}\label{f1}
For any seed homomorphism
$f_1:\Sigma_{I_0,I_1}\rightarrow\Sigma'_{I'_0,I'_1}$, we associate a
partial seed homomorphism in $\rm Hom_{par}(\Sigma,\Sigma')$, that is,
$f=id_{I'_0,I'_1}f_1:\Sigma_{I_0,I_1}\rightarrow \Sigma'$.

Conversely, for any partial seed homomorphism $f:\Sigma_{I_0,I_1}\rightarrow \Sigma'$, denoted by
$f(\Sigma_{I_0,I_1})=\Sigma'_{I'_0,I'_1}$, we can define a seed homomorphism $f_1:\Sigma_{I_0,I_1}\rightarrow\Sigma'_{I'_0,I'_1}$ via $f_1(x)=f(x)$ for all $x\in \widetilde{X}\setminus I_1$, and thus, $f=id_{I'_0,I'_1}f_1$. It is easy to see that such $f_1$ is unique for $f$.
\end{Rem}

Unfortunately, $f\not=f\circ id_{I_0,I_1}$ since usually
$Dom(f)\not=Dom(f\circ id_{I_0,I_1})$. It is similar for
$id=id_{\emptyset,\emptyset}$, so there does not exist identity in
$\rm{End_{par}}(\Sigma)$.

\begin{Prop}\label{semigroup} The associative law of the composition of partial seed homomorphisms in $\rm{End_{par}}(\Sigma)$ holds.
 Then, with the composition as multiplication, $\rm{End_{par}}(\Sigma)$ is a semigroup with zero
$\emptyset:=(\emptyset,\emptyset,\Sigma)$, including some zero-divisors.
\end{Prop}

\begin{proof}
Let $f,g,h\in \rm End_{par}(\Sigma)$ with
$\Sigma_f=\Sigma_{I_0,I_1}$, $\Sigma_g=\Sigma_{I'_0,I'_1}$ and
$\Sigma_h=\Sigma_{I''_0,I''_1}$, respectively. By the definition of composition, we have that $x\in
Dom((h\circ g)\circ f)_{ex}$ is equivalent to $x\in Dom(f)_{ex}$, $f(x)\in
Dom(g)_{ex}$ and $g\circ f(x)\in Dom(h)_{ex}$, which is also equivalent to
$x\in Dom(h\circ (g\circ f))_{ex}$. Hence, $Dom((h\circ g)\circ f)_{ex}=Dom(h\circ(g\circ f))_{ex}$.
It can be proved similarly that
$Dom((h\circ g)\circ f)_{fr}=Dom(h\circ (g\circ f))_{fr}$. Therefore, $Dom((h\circ g)\circ f)=Dom(h\circ(g\circ f))$ and then $(h\circ g)\circ f=h\circ(g\circ f)$.

Assume that $x\in X$, and
 let $I=\widetilde{X}\setminus\{x\}$. Hence, we have
$id_{\{x\},I}\circ id_{\emptyset,I}=\emptyset$, that is, $id_{\{x\},I}$
and $id_{\emptyset,I}$  are zero-divisors in
$\rm{End_{par}}(\Sigma)$.
\end{proof}

 Denote by $E(\rm{End_{par}}(\Sigma))$ the set of all
idempotents in the semigroup $\rm{End_{par}}(\Sigma)$. In
particular, $id_{I_0,I_1}=id_{I_0,I_1}\circ id_{I_0,I_1}$, so
$id_{I_0,I_1}$ is an idempotent, i.e. $id_{I_0,I_1}\in
E(\rm{End_{par}}(\Sigma))$.

However, in $\rm End_{par}(\Sigma)$, in general, an idempotent may not be in the form $id_{I_0,I_1}$. For example, for a quiver  $Q:\xymatrix{x_1\ar[r]^{}&x_2&x_3\ar[l]^{}}$,
let $f:\Sigma(Q)\rightarrow\Sigma(Q)_{\emptyset,\{x_3\}}$ be the partial seed
homomorphism such that $f(x_1)=f(x_3)=x_1$ and $f(x_2)=x_2$. It is clear that $f$
is an idempotent,  but is not in the form $id_{I_0,I_1}$.

Now recall some elementary definitions and properties on rooted cluster morphisms, refer to \cite{ADS}.

\begin{Def}(\cite{ADS})\label{rootmorph}
A {\bf rooted cluster morphism}  $f$ from $\mathcal{A}(\Sigma)$
to $\mathcal{A}(\Sigma')$ is a ring morphism which sends
1 to 1 satisfying:\\
CM1. $f(\widetilde{X})\subseteq \widetilde{X'}\sqcup \mathbb{Z}$;\\
CM2. $f(X)\subseteq X'\sqcup \mathbb{Z}$;\\
CM3. For every $(f,\Sigma,\Sigma')$-biadmissible sequence
$(y_{1},y_{2},\cdots,y_{s})$ and any $y$ in $\widetilde{X}$, we have
\\ \centerline{ $f(\mu_{y_{s}}\cdots \mu_{y_{1}}(y))=\mu_{f(y_{s})}\cdots\mu_{f(y_{1})}(f(y))$.}

\end{Def}

In \cite{ADS}, the authors defined a category \textbf{Clus} with objects are rooted cluster algebras and morphisms are rooted cluster morphisms.

\begin{Def}(\cite{HLY})\label{reducedseed}
Let $f:\mathcal{A}(\Sigma)\rightarrow \mathcal{A}(\Sigma')$ be a
rooted cluster morphism and
\begin{equation}\label{setI_1}
I_1=\{x\in
\widetilde{X}|f(x)\in\mathbb{Z}\}.
\end{equation}
From $f$,  define a
new seed $\Sigma^{(f)}=(X^{(f)},\widetilde{B^{(f)}})$ satisfying that:

 (I). $X^{(f)}=X\setminus I_1=\{x\in X|f(x)\notin \mathbb{Z}\}$;

(II). $\widetilde{X^{(f)}}=\widetilde{X}\setminus I_1=\{x\in
\widetilde{X}|f(x)\notin \mathbb{Z}\}$;

(III). $\widetilde{B^{(f)}}=(b^{(f)}_{xy})$ is a $\#(X^{(f)})\times
\#(\widetilde{X^{(f)}})$ matrix with
\begin{eqnarray*}b^{(f)}_{xy}= \left\{\begin{array}{lll}b_{xy}, & \;\text{if}\; f(z)\neq 0 \;\forall z\in I_1 \;\text{adjacent to}\; x \;\text{or}\; y;\\
0,& \text{otherwise}.\\
\end{array}\right.
\end{eqnarray*}
We call $\Sigma^{(f)}=(X^{(f)},\widetilde{B^{(f)}})$ the {\bf contraction} of $\Sigma$ under $f$.
\end{Def}

\begin{Def}\cite{HLY}
We say a rooted cluster morphism $f:\mathcal A(\Sigma)\rightarrow \mathcal A(\Sigma')$ is  noncontractible if $f(x)\neq 0$ for all $x\in \widetilde{X}$.
\end{Def}

\begin{Prop}(\cite{HLY}, Proposition 3.4) \label{induceseed}
 A rooted cluster morphism $f: \mathcal
A(\Sigma)\rightarrow\mathcal A(\Sigma')$ determines uniquely a seed
homomorphism $(f^S, \Sigma^{(f)}, \Sigma')$ from $\Sigma^{(f)}$
to $\Sigma'$ via $f^S(x)=f(x)$ for $x\in \widetilde{X^{(f)}}$.
\end{Prop}

\begin{Lem}(\cite{HLY}, Lemma 3.5)\label{equel}
Let $f,g:\mathcal{A}(\Sigma)\rightarrow\mathcal{A}(\Sigma')$
be rooted cluster morphisms. If $f(x)=g(x)\neq 0$ for all $x\in
\widetilde{X}$ of $\Sigma$, then $f=g$.
\end{Lem}

We have the following connection between rooted cluster isomorphism and seed isomorphism.

\begin{Prop}(\cite{HLY}, Proposition 3.8)\label{basiclem}
$\mathcal{A}(\Sigma)\cong\mathcal{A}(\Sigma')$ in \textbf{Clus} if
and only if $\Sigma\cong\Sigma'$ in \textbf{Seed}.
\end{Prop}

Recall that a rooted cluster morphism $f:\mathcal A(\Sigma)\rightarrow \mathcal A(\Sigma')$ is called {\bf ideal} if $\mathcal A(f^S(\Sigma^{(f)}))=f(\mathcal A(\Sigma))$. See Definition 2.11, \cite{ADS}.

\begin{Lem}(1)\label{idealmor.}(Theorem 2.25(2), \cite{CZ})
Let $f:\;\mathcal{A}(\Sigma)\rightarrow\mathcal{A}(\Sigma')$
be an ideal rooted cluster morphism. Then $f=\tau f_1$ with a surjective rooted cluster morphism $f_1$ and an injective rooted cluster morphism $\tau$, that is,
$$f:\mathcal{A}(\Sigma)\overset{f_1}{\twoheadrightarrow}\mathcal{A}(f^S(\Sigma^{(f)}))\overset{\tau}{\hookrightarrow}\mathcal{A}(\Sigma').$$

(2) $f_1$ is an ideal rooted cluster morphism.
\end{Lem}

\begin{proof}
(2) Note that $f_1^S(\Sigma^{(f_1)})=f^S(\Sigma^{(f)})$ and $f_1$ is surjective, thus we have $$\mathcal A(f_1^S(\Sigma^{(f_1)}))=\mathcal{A}(f^S(\Sigma^{(f)}))=f_1(\mathcal A(\Sigma)).$$ Therefore, $f_1$ is an ideal rooted cluster morphism.
\end{proof}

We can determinate all rooted cluster subalgebras of a given rooted cluster algebras by the theorem below.

\begin{Thm}(\cite{HLY}, Theorem 4.4)\label{rooted cluster subalgebra}
{\rm $\mathcal{A}(\Sigma')$ is a rooted cluster subalgebra of
$\mathcal{A}(\Sigma)$ if and only if there exists a mixing-type sub-seed  $\Sigma_{I_{0},I_{1}}$ of $\Sigma$ such that $\Sigma'\cong\Sigma_{I_{0},I_{1}}$ satisfies $b_{xy}=0$ for any $x\in X\setminus (I_0\cup I_1)$ and $y\in I_1$.}
\end{Thm}

\begin{Lem}(\cite{HLY}, Lemma 6.18)\label{gluingfact}
 For $y_1\not=y_2\in \widetilde{X}$, if $f(y_1)=f(y_2)\in
\widetilde{X'}$ for a rooted cluster morphism
$f:\mathcal{A}(\Sigma)\rightarrow\mathcal{A}(\Sigma')$, then $y_1,
y_2\in X_{fr}$.
\end{Lem}

Analog to $\rm Hom_{par}(\Sigma,\Sigma')$, we now
consider the set of some {\em partial} rooted cluster morphisms,
$\rm Hom_{par}(\mathcal{A}(\Sigma),\mathcal{A}(\Sigma'))$.

\begin{Def}\label{partial}
(i)~ Define a {\bf partial ideal rooted cluster morphism} $f$ from
$\mathcal{A}(\Sigma)$ to $\mathcal{A}(\Sigma')$ to be an ideal
rooted cluster morphism
$f:\mathcal{A}(\Sigma_{I_0,I_1})\rightarrow\mathcal{A}(\Sigma')$ for
a mixing-type sub-seed $\Sigma_{I_0,I_1}$ of $\Sigma$, which is
denoted as $(f,\Sigma_{I_0,I_1},\Sigma')$ or briefly as $f$. Denote
$\rm Hom_{par}(\mathcal{A}(\Sigma),\mathcal{A}(\Sigma'))$ as
 the set of all partial ideal rooted cluster morphism from
$\mathcal{A}(\Sigma)$ to $\mathcal{A}(\Sigma')$.

(ii)~ For a seed $\Sigma'$, fix a family of ideal rooted
cluster morphisms
$$\Lambda=\{h^{I_0,I_1}_{I'_0,I'_1}:\mathcal{A}(\Sigma'_{I_0,I_1})\rightarrow\mathcal{A}(\Sigma'_{I'_0,I'_1})\}_{I_0,I_0'\in X', I_1,I_1'\in
\widetilde{X'}}$$ such that $h^{J_0,J_1}_{J_0,J_1}=id_{\mathcal{A}(\Sigma'_{J_0,J_1})}$
for all $J_0,J_1$ and
$h^{J_0,J_1}_{I'_0,I'_1}\tau=h^{I_0,I_1}_{I'_0,I'_1}$ if $\Sigma'_{I_0,I_1}$ is a mixing-type sub-seed of $\Sigma'_{J_0,J_1}$ and the natural injective seed homomorphism $\Sigma'_{J_0,J_1}\rightarrow\Sigma'_{J_0,J_1}$ induces an injective rooted cluster morphism $\tau$.
  For two partial
ideal rooted cluster morphisms $$(f,\Sigma_{I_0,I_1},\Sigma')\in
\rm Hom_{par}(\mathcal{A}(\Sigma),\mathcal{A}(\Sigma')),\;\;
(g,\Sigma'_{I'_0,I'_1},\Sigma'')\in
\rm Hom_{par}(\mathcal{A}(\Sigma'),\mathcal{A}(\Sigma'')),$$ we define
the {\bf composition }
$g\circ_{\Lambda}f=gh^{J_0,J_1}_{I'_0,I'_1}f'$, where
$\Sigma'_{J_0,J_1}=f^S(\Sigma_{I_0,I_1})$ by Proposition
\ref{induceseed} and $f$ is decomposed into
$f:\mathcal{A}(\Sigma)\overset{f'}{\twoheadrightarrow}\mathcal{A}(\Sigma'_{J_0,J_1})\overset{\tau}{\hookrightarrow}\mathcal{A}(\Sigma')$
by Lemma \ref{idealmor.}.
\end{Def}

\begin{Rem}
For any mixing-type sub-seed $\Sigma'_{J_0,J_1}$, let $h^{J_0,J_1}_{J_0,J_1}=id_{\mathcal A(\Sigma'_{J_0,J_1})}$. For mixing-type sub-seeds $\Sigma'_{I_0,I_1}$ and $\Sigma'_{J_0,J_1}$, if $\mathcal A(\Sigma'_{I_0,I_1})$ is a rooted cluster subalgebra of $\mathcal A(\Sigma'_{J_0,J_1})$, then let $h^{I_0,I_1}_{J_0,J_1}$ be the natural inclusion; otherwise, let $h^{I_0,I_1}_{J_0,J_1}$ be the trivial ideal rooted cluster morphism satisfying $h^{I_0,I_1}_{J_0,J_1}(x)=1$ for any $x\in \mathcal A(\Sigma'_{I_0,I_1})$. It is easy to see that $\Lambda=\{h^{I_0,I_1}_{J_0,J_1}\}$ satisfies the conditions in Definition \ref{partial} (ii).
\end{Rem}

In this definition, note that
$g\circ_{\Lambda}f$ is ideal since the composition of ideal
rooted cluster morphisms is ideal and $g, h^{J_0,J_1}_{I'_0,I'_1}$ are ideal rooted cluster morphisms, $f'$ is also ideal by Lemma \ref{idealmor.} (2). Hence, $(g\circ_{\Lambda}f, \Sigma_{I_0,I_1}, \Sigma'')$ is a partial ideal rooted cluster morphism in $\rm Hom_{par}(\mathcal{A}(\Sigma),\mathcal{A}(\Sigma''))$.

For  $j=1,2,3$, let
$f_j:\mathcal{A}(\Sigma_{I^{j}_0,I^{j}_1})\rightarrow
\mathcal{A}(\Sigma)$ be ideal rooted cluster morphisms. Denote
$\Sigma_{{I^{j}_0}',{I^{j}_1}'}=(f_{j})^S(\Sigma^{(f_j)}_{I^{j}_0,I^{j}_1})$ and then
$f_j(\mathcal{A}(\Sigma_{I^{j}_0,I^{j}_1}))=\mathcal{A}(\Sigma_{{I^{j}_0}',{I^{j}_1}'})$. Also
denote
$\Sigma_{J_0,J_1}=(f_2\circ_{\Lambda}f_1)^{S}(\Sigma^{(f_2\circ_{\Lambda}f_1)}_{I'_0,I'_1})$ and then $\mathcal A(\Sigma_{J_0,J_1})=(f_2\circ_{\Lambda}f_1)(\mathcal
A(\Sigma_{I'_0,I'_1}))=(f_2h^{{I^{1}_0}',{I^{1}_1}'}_{I^2_0,I^2_1})(\mathcal{A}(\Sigma_{{I^{1}_0}',{I^{1}_1}'}))$.
According to Lemma \ref{idealmor.}, $f_i$ is decomposed into
$f_j:\mathcal{A}(\Sigma_{I^{j}_0,I^{j}_1})\overset{f'_j}{\twoheadrightarrow}\mathcal{A}(\Sigma_{I'^{j}_0,I'^{j}_1})\overset{\tau_j}{\hookrightarrow}
\mathcal{A}(\Sigma)$. Hence, we have the following commutative
diagram:
\[
\begin{CD}
\mathcal{A}(\Sigma_{I^1_0,I^1_1}) @>{f'_1}>{}> \mathcal{A}(\Sigma_{{I^{1}_0}',{I^{1}_1}'}) @>{\pi}>{}> \mathcal{A}(\Sigma_{J_0,J_1}) \\
@V{id}V{}V @V{h^{{I^{1}_0}',{I^{1}_1}'}_{I^2_0,I^2_1}}V{}V @V{\tau}V{}V \\
\mathcal{A}(\Sigma_{I^1_0,I^1_1})
@>{h^{{I^{1}_0}',{I^{1}_1}'}_{I^2_0,I^2_1}f'_1}>{}>
\mathcal{A}(\Sigma_{I^2_0,I^2_1}) @>{f'_2}>{}>
\mathcal{A}(\Sigma_{{I^{2}_0}',{I^{2}_1}'}),
\end{CD}
\]
 where there are an inclusion $\tau:\mathcal{A}(\Sigma_{J_0,J_1})\rightarrow\mathcal{A}(\Sigma_{{I^{2}_0}',{I^{2}_1}'})$
 and a surjective morphism
$\pi:\mathcal{A}(\Sigma_{{I^{1}_0}',{I^{1}_1}'})\rightarrow\mathcal{A}(\Sigma_{J_0,J_1})$
by Lemma \ref{idealmor.} since
$f'_2h^{{I^{1}_0}',{I^{1}_1}'}_{I^2_0,I^2_1}$ is ideal. By
definition of composition, we have
$$(f_3\circ_{\Lambda}f_2)\circ_{\Lambda}f_1=f_3h^{{I^{2}_0}',{I^{2}_1}'}_{I^{3}_0,I^{3}_1}f'_2h^{{I^{1}_0}',{I^{1}_1}'}_{I^{2}_0,I^{2}_1}f'_1,$$
$$f_3\circ_{\Lambda}(f_2\circ_{\Lambda}f_1)=f_3h^{J_0,J_1}_{I^{3}_0,I^{3}_1}\pi f'_1.$$
By Definition \ref{partial},
$h^{J_0,J_1}_{I^3_0,I^3_1}=h^{{I^{2}_0}',{I^{2}_1}'}_{I^3_0,I^3_1}\tau$;
hence, we have
\begin{equation}\label{associativity}
f_3\circ_{\Lambda}(f_2\circ_{\Lambda}f_1)=f_3h^{J_0,J_1}_{I^{3}_0,I^{3}_1}\pi
f'_1=f_3h^{{I^{2}_0}',{I^{2}_1}'}_{I^3_0,I^3_1}\tau\pi
f'_1=f_3h^{{I^{2}_0}',{I^{2}_1}'}_{I^3_0,I^3_1}f'_2h^{{I^{1}_0}',{I^{1}_1}'}_{I^{2}_0,I^{2}_1}f'_1=(f_3\circ_{\Lambda}f_2)\circ_{\Lambda}f_1,
\end{equation}
where the second equality is due to Definition \ref{partial} (ii).  This means the associative law holds for the composition $\circ_{\Lambda}$.

Analog to Proposition \ref{semigroup}, we have the similar result
on the following.

\begin{Prop}\label{semi-groupforalg} The associative law of the composition $\circ_{\Lambda}$ of partial ideal rooted cluster morphisms in $\rm{End_{par}}(\mathcal A(\Sigma))$ holds.
 Then, with $\circ_{\Lambda}$ as multiplication, $\rm{End_{par}}(\mathcal A(\Sigma))$ is a semigroup with zero
the zero-morphism, including some zero-divisors.
\end{Prop}

In the sequel,  we will use the theory of seed homomorphisms to discuss on the Green's equivalences of
$\rm End_{par}(\mathcal{A}(\Sigma))$ with
$\rm End_{par}(\Sigma)$.

\section{ Green's equivalences in semigroup of partial seed homomorphisms}

It is well known that the theory of Green's equivalences plays a significant role in the algebraic theory of semigroups.
They are concerned with mutual divisibility of various kinds, and all of them reduce to the universal
equivalence in a group.
For the use in this part, we need to introduce a fundamental knowledge of the theory of Green's equivalences and regular semigroups referring to \cite{H}.

Given a semigroup $S$ with a multiplication $\cdot$, there are five equivalent relations, $\mathcal L, \mathcal R, \mathcal H, \mathcal D$ and $\mathcal J$, called {\bf Green's equivalences} on $S$. More precisely,

$\mathcal{L}$ on $S$ is defined by the rule that
 for $x,y\in S$, $x\mathcal{L}y$ if $Sx\cup \{x\}=Sy\cup \{y\}$ (left ideals).

  $\mathcal{R}$ on $S$ is defined by the rule that for $x,y\in S$, $x\mathcal{R}y$ if $xS\cup \{x\}=yS\cup \{y\}$ (right ideals).

$\mathcal{H}=\mathcal{L}\wedge\mathcal{R}$, the intersection of $\mathcal L$ and $\mathcal R$, that is, for $x,y\in S$, $x\mathcal{H}y$
if $x\mathcal{L}y$ and $x\mathcal{R}y$.

$\mathcal{D}=\mathcal{L}\vee\mathcal{R}$, the join of $\mathcal L$ and $\mathcal R$.
It is known that $\mathcal{D}=\mathcal{L}\circ\mathcal{R}=\mathcal{R}\circ\mathcal{L}$. Then
 for $x,y\in S$, $x\mathcal{D}y$ if and only if $\exists\; z\in S$ such
that $x\mathcal{L}z$ and $z\mathcal{R}y$ and if and only if $\exists\; w\in S$ such that $x\mathcal{R}w$ and
$w\mathcal{L}y$.

 $\mathcal J$ on $S$ is defined by the rule that for $x,y\in S$,  $x\mathcal{J}y$ if $SxS\cup xS\cup
Sx\cup\{x\}=SyS\cup yS\cup Sy\cup\{y\}$ as two-side ideals.

 The equivalent classes of $\mathcal L, \mathcal R, \mathcal H, \mathcal D$ and $\mathcal J$ are called {\bf  $\mathcal{L}$-classes, $\mathcal{R}$-classes,
 $\mathcal{H}$-classes, $\mathcal{D}$-classes and $\mathcal{J}$-classes}. For $x\in S$, the equivalent class, including $x$, is written as
 $L_x, R_x, H_x, D_x$ and $J_x$  respectively.

From the definition, we have the relations among these equivalent classes as that: $\mathcal H\subseteq \mathcal L\subseteq\mathcal D$,  $\mathcal H\subseteq \mathcal R\subseteq\mathcal D$ and $\mathcal D\subseteq \mathcal J$.

The following results are well-known.

\begin{Prop}\label{H}
(\cite{H}) If $H$ is a $\mathcal{H}$-class in a semigroup $S$, then
either $H^2\cap H=\emptyset$ or $H^2=H$ and $H$ is a subgroup of
$S$.
\end{Prop}

\begin{Prop}\label{idem2}
(\cite{H}) Let $x$ and $y$ be two idempotents in a semigroup $S$.
Then $x\mathcal{D}y$ if and only if there exists an element $a\in S$
and an inverse $a'$ of $a$ such that $aa'=x$ and $a'a=y$.
\end{Prop}

An element $x$ of a semigroup $S$ is called {\bf regular} if there exists another
element $y\in S$ such that $xyx=x$. The semigroup $S$ is called a {\bf regular semigroup}
if all of its elements are regular. We have the following
propositions.

\begin{Prop}\label{regular}
(\cite{H}) If $x$ is a regular element in a semigroup $S$, then every element of
$D_x$ is regular. In this case, the $\mathcal D$-class $D_x$ is called {\bf regular}.
\end{Prop}

\begin{Prop}\label{idempotent}
(\cite{H}) In a regular $\mathcal{D}$-class  of a semigroup $S$, each $\mathcal{L}$-class
and each $\mathcal{R}$-class contain at least one idempotent.
\end{Prop}

By Proposition \ref{semigroup}, $\rm{End_{par}}(\Sigma)$ is a
semigroup with zero element $\emptyset$. We will characterize the
classification of mixing-type sub-seeds of a seed under isomorphisms and
then give the classification of sub-rooted cluster algebras in a
rooted cluster algebra using the method of Green's equivalences of
semigroups as mentioned above.

In general, the semigroup $\rm{End_{par}}(\Sigma)$ is not regular.
For example, if $\Sigma$ is the seed associate with the
quiver $Q:\xymatrix{x_1\ar[r]^{}&x_2 &x_3\ar@{=>}[l]}$, then the partial seed
homomorphism $f$ on $\Sigma$, satisfying $f(\xymatrix{x_1\ar[r]^{}&x_2})=\xymatrix{x_2
&x_3\ar@{=>}[l]}$ and $f(x_1)=x_3$ and $f(x_2)=x_2$, is not
regular. Indeed, if $f$ is regular, there exists $g\in
\rm End_{par}(\Sigma)$ such that $f\circ g\circ f=f$. From this, it follows that $g(x_2)=x_2$ and
$g(x_3)=x_1$, and then $|b_{23}|=2>1=|b_{g(2)g(3)}|=|b_{21}|$, which contradicts the fact that $g$ is a partial seed homomorphism.

By Proposition \ref{H}, the $\mathcal{H}$-class $H_{id_{I_0,I_1}}$
is a group with identity $id_{I_0,I_1}$  under the multiplication of
$\rm{End_{par}}(\Sigma)$. According to Proposition \ref{regular},
the $\mathcal{D}$-class $D_{id_{I_0,I_1}}$
 is regular since $id_{I_0,I_1}$ is
regular.

Denote by $Aut(\Sigma_{I_0,I_1})$  the group of automorphisms of the mixing-type
sub-seed $\Sigma_{I_0,I_1}$ in {\bf Seed}.

\begin{Lem}\label{from2.22}
For $f,g\in \rm{End_{par}}(\Sigma)$, (a)~ if $f\mathcal{L}g$, then $\Sigma_{f}=\Sigma_{g}$; (b)~ if $f\mathcal{R}g$, then $f(\Sigma_{f})=g(\Sigma_{g})$.
\end{Lem}

\begin{proof}
(a)~ If $f\mathcal{L}g$, then $h\circ f=g$ and $h'\circ g=f$ for some $h,h'\in \rm{End_{par}}(\Sigma)$. By Fact \ref{dom} (a), $\Sigma_{g}$ is a  mixing-type sub-seed of $\Sigma_{f}$ since $h\circ f=g$. Similarly, $\Sigma_{f}$ is a  mixing-type sub-seed of $\Sigma_{g}$. Thus, $\Sigma_{f}=\Sigma_{g}$.

(b)~ If $f\mathcal{R}g$, then $f\circ h=g$ and $g\circ h'=f$ for some $h,h'\in \rm{End_{par}}(\Sigma)$. By Fact \ref{dom} (b), $g(\Sigma_{g})$ is a mixing-type sub-seed of $f(\Sigma_{f})$ since  $f\circ h=g$. Similarly, $f(\Sigma_{f})$ is a mixing-type sub-seed of $g(\Sigma_{g})$. Thus, $f(\Sigma_{f})=g(\Sigma_{g})$.
\end{proof}

\begin{Lem}\label{idem}
\rm{In the semigroup $\rm{End_{par}}(\Sigma)$,  for
$I_0,I'_0\subseteq X$ and $I_1,I'_1\subseteq \widetilde{X}$, the following
statements hold:

(1)~ If $\Sigma_{I_0,I_1}\overset{\varphi}{\cong}\Sigma_{I'_0,I'_1}$
in {\bf Seed} and $f:=id_{I'_0,I'_1}\varphi$ and
$f^{-1}:=id_{I_0,I_1}\varphi^{-1}\in \rm{End_{par}}(\Sigma)$, then
 $f\mathcal {L} id_{I_0,I_1}, f^{-1}\mathcal {R} id_{I_0,I_1}$ and $f\mathcal {R} id_{I'_0,I'_1}, f^{-1}\mathcal {L} id_{I'_0,I'_1}$.

(2)~ $\Sigma_{I_0,I_1}\cong \Sigma_{I'_0,I'_1}$ in {\bf Seed} if and
only if $id_{I_0,I_1}\mathcal{D} id_{I'_0,I'_1}$.

(3)~ $H_{id_{I_0,I_1}}\cong Aut(\Sigma_{I_0,I_1})$.}
\end{Lem}

\begin{proof}
(1)~ As $f\circ id_{I_0,I_1}=f$ and $f^{-1}\circ f=id_{I_0,I_1}$, we have
$f\mathcal {L} id_{I_0,I_1}$. Similarly, $f^{-1}\mathcal {R}
id_{I_0,I_1}$.

Dually, we can get $f\mathcal {R} id_{I'_0,I'_1}$ and $f^{-1}\mathcal {L}
id_{I'_0,I'_1}$.

(2)~ ``Only if":\;  Assume that
$\Sigma_{I_0,I_1}\overset{\varphi}{\cong}\Sigma_{I'_0,I'_1}$, by
(1), $f\mathcal{L}id_{I_0,I_1}$ and $id_{I'_0,I'_1}\mathcal{R}f$.
Then, $id_{I_0,I_1}\mathcal{D}id_{I'_0,I'_1}$.

``If":\; If $id_{I_0,I_1}\mathcal{D}id_{I'_0,I'_1}$, by Proposition
\ref{idem2}, there exist $f, g\in\rm{End}_{par}(\Sigma)$ such that
$g\circ f=id_{I_0,I_1}$ and $f\circ g=id_{I'_0,I'_1}$. By Lemma \ref{from2.22},
$f(\Sigma_{I_0,I_1})=\Sigma_{I'_0,I'_1}$ and
$f_1:\Sigma_{I_0,I_1}\rightarrow \Sigma_{I'_0,I'_1}$ is an
isomorphism, where $f_1(x)=f(x)$ for all $x\in \widetilde{X}\setminus I_1$.

(3)~ We have $H_{id_{I_0,I_1}}$ as a group with identity
$id_{I_0,I_1}$. According to (1), for $I_0=I'_0$ and $I_1=I'_1$ and
$\varphi\in \rm Aut(\Sigma_{I_0,I_1})$, we have $id_{I_0,I_1}\varphi\mathcal H
id_{I_0,I_1}$ and then $id_{I_0,I_1}\varphi\in H_{id_{I_0,I_1}}$. So,
$\rm Aut(\Sigma_{I_0,I_1})\rightarrow H_{id_{I_0,I_1}}: \varphi\longmapsto id_{I_0,I_1}\varphi$ is an injective group homomorphism.

Conversely, assume $f\in H_{id_{I_0,I_1}}$. Then
$f\mathcal L id_{I_0,I_1}$ and $f\mathcal R id_{I_0,I_1}$. It
follows that there exists $g,g'\in H_{id_{I_0,I_1}}$ such
that $f\circ g=id_{I_0,I_1}$ and $g'\circ f=id_{I_0,I_1}$, and then, by Lemma \ref{from2.22} (b), we have
 $f(\Sigma_{I_0,I_1})=\Sigma_{I_0,I_1}$. Since $g\in H_{id_{I_0,I_1}}$, then $g\mathcal L id_{I_0,I_1}$. By Lemma \ref{from2.22} (a), $\Sigma_g=\Sigma_{I_0,I_1}$. Similarly, $\Sigma_{g'}=\Sigma_{I_0,I_1}$. Then,
we have $g(\Sigma_{I_0,I_1})=\Sigma_{I_0,I_1}$ and
$g'(\Sigma_{I_0,I_1})=\Sigma_{I_0,I_1}$.

Owing to $f\circ g=id_{I_0,I_1}$ and $g'\circ f=id_{I_0,I_1}$, we have,
respectively, the maps of sets: $$(f\circ g)|_{\widetilde{X}\setminus
I_1}=id_{\widetilde{X}\setminus I_1}\;\;\;\;
\text{and}\;\;\;\;(f\circ g)|_{X\setminus (I_0\cup I_1)}=id_{X\setminus
(I_0\cup I_1)},$$
$$(g'\circ f)|_{\widetilde{X}\setminus I_1}=id_{\widetilde{X}\setminus I_1}\;\;\;\;  \text{and}\;\;\;\;(g'\circ f)|_{X\setminus (I_0\cup I_1)}=id_{X\setminus (I_0\cup I_1)}.$$
It follows that $g|_{\widetilde{X}\setminus
I_1}=g'|_{\widetilde{X}\setminus I_1}$ and $g|_{X\setminus (I_0\cup
I_1)}=g'|_{X\setminus (I_0\cup I_1)}$.
  Hence, by the definition of seed homomorphisms, we have
$g=g':\Sigma_{I_0,I_1}\rightarrow\Sigma$, which means
$f_1:\Sigma_{I_0,I_1}\rightarrow
f(\Sigma_{I_0,I_1})=\Sigma_{I_0,I_1}\in \rm Aut(\Sigma_{I_0,I_1})$
and then,
$\rm Aut(\Sigma_{I_0,I_1})\rightarrow H_{id_{I_0,I_1}}: \varphi\longmapsto id_{I_0,I_1}\varphi$ is a surjective group homomorphism. So, (1) holds.
\end{proof}

\begin{Lem}\label{Rclasswithidentity}
\rm{ Let $R$ be an $\mathcal{R}$-class of
$\rm{End}_{par}(\Sigma)$. Then $R$ is regular if and only if there is an
$(I_0,I_1)$-type sub-seed of $\Sigma$ such that $id_{I_0,I_1}\in R$. }
\end{Lem}

\begin{proof}
Assume that $id_{I_0,I_1}\in R$ for some $I_0, I_1$. Since $id_{I_0,I_1}$ is idempotent, it is regular, and also, $R$ is regular by Proposition
\ref{regular}.

Conversely, assume $R$ is regular. By Proposition \ref{idempotent},
there exists an idempotent partial seed homomorphism
$(f,\Sigma_{I'_0,I'_1}, \Sigma)$ in $R$. Denote
$f(\Sigma_{I'_0,I'_1})=\Sigma_{I_0,I_1}$. Hence, as a map of
sets, $f:\widetilde{X}\setminus I_1'\rightarrow
\widetilde{X}\setminus I_1$ is surjective and $f=f|_{(\widetilde{X}\setminus I_1')\cap
(\widetilde{X}\setminus I_1)}\circ f$. It
follows that $$\#( (\widetilde{X}\setminus I_1')\cap
(\widetilde{X}\setminus I_1))\geq
\#(Im(f^{2}))=\#(Im(f))=\#(\widetilde{X}\setminus I_1).$$ Thus,
$(\widetilde{X}\setminus I_1')\cap (\widetilde{X}\setminus
I_1)=\widetilde{X}\setminus I_1$, which means
$\widetilde{X}\setminus I_1\subseteq \widetilde{X}\setminus I_1'$.

For any $y\in \widetilde{X}\setminus I_1$, there exists $x\in
\widetilde{X}\setminus I_1'$ such that $f(x)=y$. So,
$f(y)=f(f(x))=f(x)=y$; thus, we have $f|_{\widetilde{X}\setminus
I_1}=id_{\widetilde{X}\setminus I_1}$. Therefore,
$f|_{\Sigma_{I_0,I_1}}=id_{I_0,I_1}$.  Then,
$f\circ id_{I_0,I_1}=id_{I_0,I_1}$.

Now we show $id_{I_0,I_1}\circ f=f$. Indeed, it suffices to prove that $Dom(id_{I_0,I_1}\circ f)=Dom(f)$. It is easy to see that $Dom(id_{I_0,I_1}\circ f)_{ex}=\{x\in Dom(f)_{ex}\;|\;f(x)\in Dom(id_{I_0,I_1})_{ex}\}=Dom(f)_{ex},$ If $Dom(id_{I_0,I_1}\circ f)_{fr}\subseteq Dom(f)_{fr}$, then there exists a $x\in Dom(f)_{fr}$ such that $f(x)\in Dom(id_{I_0,I_1})_{ex}$. According to the definition of image seed, it means there exists a $y\in Dom(f)_{ex}$ such that $f(y)=f(x)$. Since $f\circ f=f$, if $f(x)=f(y)$, then $x,y\in Dom(f)_{ex}$ or $x,y\in Dom(f)_{fr}$. It contradicts to $x\in Dom(f)_{fr}, y\in Dom(f)_{ex}$. Therefore, we have $Dom_(id_{I_0,I_1}\circ f)_{fr}=Dom(f)_{fr}$.

Thus, we have
$f\mathcal{R}id_{I_0,I_1}$. Hence, $id_{I_0,I_1}\in R$.
\end{proof}

\begin{Cor}\label{id}
 For a rooted cluster algebra $\mathcal A(\Sigma)$, a
$\mathcal{D}$-class $D$ in the semigroup $\rm{End_{par}}(\Sigma)$ is
regular if and only if there exists an
$(I_0,I_1)$-type sub-seed of $\Sigma$ such that
$id_{I_0,I_1}\in D$.
\end{Cor}

A seed homomorphism $f:\Sigma_1\rightarrow\Sigma_2$ is called
a {\bf retraction} if there exists a seed homomorphism
$g:\Sigma_2\rightarrow\Sigma_1$ such that $fg=id_{\Sigma_2}$.

\begin{Lem}\label{Rclass}
Let $f\in \rm{End}_{par}(\Sigma)$ from $\Sigma_{I'_0,I'_1}$. Then, for a mixing-type sub-seed
$\Sigma_{I_0,I_1}$ of $\Sigma$, $f\mathcal{R}id_{I_0,I_1}$ if and only if the following statements hold:

(a)~ $f(\Sigma_{I'_0,I'_1})=\Sigma_{I_0,I_1}$;

(b)~ $f_1:\Sigma_{I'_0,I'_1}\rightarrow \Sigma_{I_0,I_1}$ is a
retraction from $\Sigma_{I'_0,I'_1}$ to $\Sigma_{I_0,I_1}$, where $f_1$ is defined from $f$ by Remark \ref{f1};

(c)~ For $x,y\in Dom(f)$, if $f(x)=f(y)$, then either $x,y\in Dom(f)_{ex}$ or $x,y\in Dom(f)_{fr}$.
\end{Lem}
\begin{proof}

By Remark \ref{f1}, $f_1$ is defined to satisfy $f_1(x)=f(x)$ for all $x\in \widetilde{X}\setminus I'_1$.

``If'':\; Since $f_1$ is a retraction, there exists
$g_1:\Sigma_{I_0,I_1}\rightarrow\Sigma_{I'_0,I'_1}$ such that
$id_{\Sigma_{I_0,I_1}}=f_1g_1$. If $g=id_{I'_0,I'_1}g_1$, then $f\circ g=id_{I_0,I_1}f_1\circ id_{I'_0,I'_1}g_1=id_{I_0,I_1}f_1g_1=id_{I_0,I_1}$, where the second equality is due to the definition of composition of $f$ and $g$.

Since $$Dom(id_{I_0,I_1}\circ f)_{ex}=\{x\in Dom(f)_{ex}\;|\; f(x)\in Dom(id_{I_0,I_1})_{ex}\}\overset{\text{by}\;(a)}{=}Dom(f)_{ex},$$
$$Dom(id_{I_0,I_1}\circ f)_{fr}=\{x\in Dom(f)_{fr}\;|\; f(x)\in Dom(id_{I_0,I_1})_{fr}\}\overset{\text{by}\;(c)}{=}Dom(f)_{fr}$$
 and $(id_{I_0,I_1}\circ f)(x)=id_{I_0,I_1}(f_1(x))=f(x)$ for all $x\in Dom(f)$.
 Hence, we have $f=id_{I_0,I_1}\circ f$. Thus, $f\mathcal{R}id_{I_0,I_1}$.

``Only If''\;\; If $f\mathcal{R}id_{I_0,I_1}$, then $f\circ g=id_{I_0,I_1}$
and $id_{I_0,I_1}\circ h=f$ for some $g,h\in \rm{End}_{par}(\Sigma)$. By Lemma \ref{from2.22} (b), we have $f(\Sigma_{I'_0,I'_1})=(id_{I_0,I_1})(\Sigma_{I_0,I_1})=\Sigma_{I_0,I_1}$, i.e. (a) holds.
  As we have $id_{I_0,I_1}=f\circ g=id_{I_0,I_1}f_1g_1$, hence, $id_{I_0,I_1}id_{\Sigma_{I_0,I_1}}=id_{I_0,I_1}f_1g_1$. Then by Remark \ref{f1} (b), it  follows that   $f_1g_1=id_{\Sigma_{I_0,I_1}}$,  i.e. (b) holds.

   For $x,y\in \widetilde{X}\setminus I'_1=Dom(f)$, let $f(x)=f(y)$.
   Assume that $x\in Dom(f)_{fr}$ and $y\in Dom(f)_{ex}$.
   Since $id_{I_0,I_1}\circ h=f$, by Fact \ref{dom} (a), we have $x\in Dom(h)_{fr}$ and $y\in Dom(h)_{ex}$. Moreover, $h(x)=id_{I_0,I_1}(h(x))=f(x)=f(y)=id_{I_0,I_1}(h(y))=h(y)$.
  Owing to the above assumption, we have $x\in Dom(id_{I_0,I_1}\circ h)_{fr}$ and $y\in Dom(id_{I_0,I_1}\circ h)_{ex}$. Then by (\ref{ex-f}),
    $h(x)\in Dom(id_{I_0,I_1})_{fr}$ and $h(y)\in Dom(id_{I_0,I_1})_{ex}$. But, it contradicts with $h(x)=h(y)$. Therefore, the above assumption is not true.
    Similarly, $x\in Dom(f)_{ex}$ and $y\in Dom(f)_{fr}$ are also impossible. Hence, (c) follows.
\end{proof}

\begin{Lem}\label{identity}
For any mixing-type sub-seeds $\Sigma_{I_0,I_1}$ and $\Sigma_{I'_0,I'_1}$
of $\Sigma$, it holds that\\
(1)~ $id_{I_0,I_1}\mathcal{L}id_{I'_0,I'_1}$ if and only if
$\Sigma_{I_0,I_1}=\Sigma_{I'_0,I'_1}$;\\
(2)~ $id_{I_0,I_1}\mathcal{R}id_{I'_0,I'_1}$ if and only if
$\Sigma_{I_0,I_1}=\Sigma_{I'_0,I'_1}$.
\end{Lem}

\begin{proof}
It follows immediately by Lemma \ref{from2.22}.
\end{proof}

Let $f,f'\in
\rm{End}_{par}(\mathcal{A}(\Sigma))$ be two noncontractible partial
rooted cluster morphisms. Now assume that the restricted partial
seed homomorphisms $f^S$ and $f'^S$ in $\rm{End}_{par}(\Sigma)$ via
Proposition \ref{induceseed} are regular. Then we have the following.

\begin{Lem}\label{retraction}
Let $(f,\Sigma_{I_0,I_1},\Sigma)\in
\rm{End}_{par}(\mathcal{A}(\Sigma))$ be a noncontractible ideal
partial rooted cluster morphism with $f^S$ the restricted partial
seed homomorphism by Proposition \ref{induceseed}. Denote
$f^S(\Sigma_{I_0,I_1})=\Sigma_{J_0,J_1}$. If
$f^S\mathcal{R}id_{J_0,J_1}$ in $\rm{End}_{par}(\Sigma)$, then there
exists an ideal rooted cluster homomorphism
$g:\mathcal{A}(\Sigma_{J_0,J_1})\rightarrow
\mathcal{A}(\Sigma_{I_0,I_1})$ such that
$f\circ_{\Lambda}g=id_{\mathcal{A}(\Sigma_{J_0,J_1})}$ and thus,
$f\mathcal{R}id_{\mathcal{A}(\Sigma_{J_0,J_1})}$ in
$\rm{End}_{par}(\mathcal{A}(\Sigma))$.
\end{Lem}

\begin{proof}
Since $f$ is ideal, we have $\mathcal A(\Sigma_{J_0,J_1})=\mathcal
A(f^S(\Sigma_{I_0,I_1}))=f(\mathcal A(\Sigma_{I_0,I_1}))$. Owing to
$f^S\mathcal{R}id_{J_0,J_1}$, by Lemma \ref{Rclass},
$f_1^S:\Sigma_{I_0,I_1}\rightarrow\Sigma_{J_0,J_1}$ is a retraction, where $f_1(x)=f(x)$ for all $x\in \widetilde{X}\setminus I'_1$.
Hence, there exists a seed homomorphism
$g^{S'}:\Sigma_{J_0,J_1}\rightarrow\Sigma_{I_0,I_1}$ such that
$f_1^Sg^{S'}=id_{\Sigma_{J_0,J_1}}$, which means that $g^{S'}$gives a
bijection between the set of variables of $\Sigma_{J_0,J_1}$ and
that of its image seed. Owing to $g^{S'}((X_{fr}\cup
J_0)\setminus J_1)\subseteq (X_{fr}\cup I_0)\setminus I_1$,
$g^{S'}(\Sigma_{J_0,J_1})$ is a mixing-type sub-seed of
$\Sigma_{I_0,I_1}$ of $(\emptyset, I'_1)$ type for $$I'_1=\{x\in
\widetilde{X}\setminus I_1\mid x\not\in g^{S'}(\widetilde{X}\setminus
J_1)\}.$$ For any $x,y\in \widetilde{X}\setminus J_1$, we have
$|b_{xy}|\leq|b_{g^{S'}(x)g^{S'}(y)}|\leq
|b_{f_1^{S}g^{S'}(x)f_1^{S}g^{S'}(y)}|=|b_{xy}|$. Thus,
$|b_{xy}|=|b_{g^{S'}(x)g^{S'}(y)}|$. Therefore, by definition of seed
isomorphisms, we have
\begin{equation}\label{usefuliso}
\Sigma_{J_0,J_1}\cong (\Sigma_{I_0,I_1})_{\emptyset,
I'_1}=g^{S'}(\Sigma_{J_0,J_1}),
\end{equation}
 so $g^{S'}$ is injective.

 Now we will use $g^{S'}$ to induce an injective rooted cluster
morphism $g$.
\\
{\bf Claim}:\; $\mathcal{A}((\Sigma_{I_0,I_1})_{\emptyset, I'_1})$ is a rooted cluster subalgebra $\mathcal{A}(\Sigma_{I_0,I_1})$.

In fact, for any $x\in I'_1$, we have $g^{S'}f_1^{S}(x)\neq x$
since $g^{S'}f_1^{S}(x)\in g^{S'}(\widetilde{X}\setminus J_1)$, whereas
$x\not\in g^{S'}(\widetilde{X}\setminus J_1)$. For any $y\in
X\setminus (I_0\cup I_1\cup I'_1)$, by CM3, we have
$f(\mu_{y}(y))=\mu_{f(y)}(f(y))$.  Moreover, since
$f(g^{S'}f_1^{S}(x))=(f_1^{S}g^{S'})(f_1^{S}(x))=f_1^{S}(x)$,  comparing
the exponents of $f(y)$ in both sides of
$f(\mu_{y}(y))=\mu_{f(y)}(f(y))$, we get
\begin{equation}\label{great}
|b_{y(g^{S'}f_1^{S})(x)}|+|b_{yx}|\leq \sum\limits_{z\in
\widetilde{X}\setminus
I_1,f(z)=f(x)}|b_{yz}|=|b_{f^{S}(y)f^{S}(x)}|,
\end{equation}
As $f(g^{S'}f_1^{S}(y))=f(y)$, we have $g^{S'}f^{S}(y)=y$ since two
exchangeable variables can not have the same image for a rooted
cluster morphism according to CM3. Furthermore, since $f_1^{S}(y)$ and $f_1^{S}(x)
\in \widetilde{X}\setminus J_1$, we obtain
\begin{equation}\label{less}
|b_{f_1^{S}(y)f_1^{S}(x)}|=|b_{(g^{S'}f_1^{S})(y)(g^{S'}f_1^{S})(x)}|=|b_{y(g^{S'}f_1^{S})(x)}|.
\end{equation}
Combining (\ref{great}) and (\ref{less}), we obtain that $b_{yx}=0$.
By Theorem \ref{rooted cluster subalgebra},
$\mathcal{A}((\Sigma_{I_0,I_1})_{\emptyset, I'_1})$ is a rooted
cluster subalgebra of $\mathcal{A}(\Sigma_{I_0,I_1})$.

Following this claim, we have an injective rooted cluster morphism
$\iota:\mathcal{A}((\Sigma_{I_0,I_1})_{\emptyset, I'_1})\rightarrow
\mathcal{A}(\Sigma_{I_0,I_1})$.

Let $\mathcal{A}(\Sigma_{J_0,J_1})\overset{g'}{\cong}\mathcal{A}((\Sigma_{I_0,I_1})_{\emptyset,
I'_1})$ be the rooted cluster isomorphism obtained from the seed isomorphism (\ref{usefuliso}).

We then obtain an injective rooted cluster morphism $g=\iota
g':\mathcal{A}(\Sigma_{J_0,J_1})\rightarrow
\mathcal{A}(\Sigma_{I_0,I_1})$.

Since $f_1^Sg^{S'}=id_{\Sigma_{J_0,J_1}}$, we get $f_1\iota
g'|_{\widetilde{X}\setminus J_1}=id_{\widetilde{X}\setminus J_1}$;
thus, it follows $f_1\iota g'=id_{\mathcal{A}(\Sigma_{J_0,J_1})}$.

 Denote $(\Sigma_{I_0,I_1})_{\emptyset,
I'_1}=\Sigma_{K_0,K_1}$ for $K_0=I_0\setminus I'_1$ and $K_1=I_1\cup
I'_1$.

Since $\iota:
\mathcal{A}(\Sigma_{K_0,K_1})\rightarrow\mathcal{A}(\Sigma_{I_0,I_1})$
is injective,
$h^{K_0,K_1}_{I_0,I_1}=h^{I_0,I_1}_{I_0,I_1}\iota=id_{\mathcal{A}(\Sigma_{I_0,I_1})}\iota$
 by Definition \ref{partial} (ii).
Thus, we have $f\circ_{\Lambda}g=fh^{K_0,K_1}_{I_0,I_1}g'=f\iota
g'=id_{\mathcal{A}(\Sigma_{J_0,J_1})}$.

By Lemma \ref{idealmor.},
$f:\mathcal{A}(\Sigma_{I_0,I_1})\overset{f_1}{\twoheadrightarrow}\mathcal{A}(\Sigma_{J_0,J_1})\overset{\tau}{\hookrightarrow}\mathcal{A}(\Sigma)$ with $f_1$ and $\tau$, respectively, the surjective and injective
rooted cluster morphisms. Then, it holds
$id_{\mathcal{A}(\Sigma_{J_0,J_1})}\circ_{\Lambda}
f=id_{\mathcal{A}(\Sigma_{J_0,J_1})}h_{J_0,J_1}^{J_0,J_1}
f_1=id_{\mathcal{A}(\Sigma_{J_0,J_1})}id_{\mathcal{A}(\Sigma_{J_0,J_1})}g
f_1=\tau id_{\mathcal{A}(\Sigma_{J_0,J_1})}f_1=\tau f_1=f$. Hence,
$f\mathcal{R}id_{\mathcal{A}(\Sigma_{J_0,J_1})}$ in
$\rm{End}_{par}(\mathcal{A}(\Sigma))$.
\end{proof}

\section{Proofs of Theorem \ref{greenrelations}, Theorem \ref{greens} and Theorem \ref{number}}

\subsection{Green's equivalences in
$\rm{End_{par}}(\Sigma)$}

Now we can obtain the characterization of Green's equivalences in
$\rm{End_{par}}(\Sigma)$, as follows as given in Theorem \ref{greenrelations} except of the $\mathcal J$-relation.

{\bf Proof of Theorem \ref{greenrelations}.}

\begin{proof}
Denote the image seeds as
$\Sigma_{J_0,J_1}=f(\Sigma_{I_0,I_1})$ and
$\Sigma_{J'_0,J'_1}=f'(\Sigma_{I'_0,I'_1})$.

(1)~ ``If'':\; According to Lemma \ref{Rclasswithidentity} and Lemma \ref{Rclass}, since $f,f'$ are regular, we have
$f\mathcal{R}id_{J_0,J_1}$ and $f'\mathcal{R}id_{J_0,J_1}$; thus,
$f\mathcal{R}f'$.

``Only If'':\; By Lemma \ref{Rclasswithidentity}, there exists
$id_{I''_0,I''_1}$ such that $f\mathcal{R}id_{I''_0,I''_1}$ and
$f'\mathcal{R}id_{I''_0,I''_1}$ hold. By Lemma
\ref{Rclass}, (a) holds.

(2)~ ``If'':\; We have
$\Sigma_{J_0,J_1}\overset{g}{\cong}\Sigma_{J'_0,J'_1}$ and $f'=(id_{J'_0,J'_1}g)\circ f$. Thus, $(id_{J_0,J_1}g^{-1})\circ f'=(id_{J_0,J_1}g^{-1})\circ((id_{J'_0,J'_1}g)\circ f)=f$. Therefore,
$f\mathcal{L}f'$.

``Only If'':\; Owing to $f\mathcal{L} f'$, by Lemma \ref{from2.22}, (c) follows.

From $f'=g\circ f$ and $f=g'\circ f'$, we get $f=g'\circ g\circ f$ and $f'=g\circ g'\circ f'$. Denote $\Sigma_g=\Sigma_{K_0,K_1}$ and $\Sigma_{g'}=\Sigma_{K'_0,K'_1}$. Let $g=id_{K_0,K_1}g_1$ and $g'=id_{K'_0,K'_1}g'_1$. By the definition of composition of partial seed homomorphisms, for any $y\in
\widetilde{X}\setminus K_1$, there exists $x\in
\widetilde{X}\setminus I_1$ such that $y=f(x)$; thus,
$g'_1g_1(y)=g'g(y)=g'g(f(x))=f(x)=y$. Hence, we have
$g'_1g_1=id_{\widetilde{X}\setminus K_1}$. Similarly,
$g_1g'_1=id_{\widetilde{X}\setminus K'_1}$. It means $g'_1g_1=id_{\Sigma_{K_0,K_1}}$ and $g_1g'_1=id_{\Sigma_{K'_0,K'_1}}$. By
Lemma \ref{isomorphism},
$g_1:\Sigma_{J_0,J_1}\rightarrow\Sigma_{J'_0,J'_1}$ is a seed
isomorphism, that is, (e) follows.

(3)~ It follows directly from the definition of $\mathcal H$-relation and by (1) and (2).

(4)~ According to Lemma \ref{Rclasswithidentity} and Lemma \ref{Rclass}, since $f,f'$ are regular, we have
$f\mathcal{R}id_{J_0,J_1}$ and $f'\mathcal{R}id_{J_0,J_1}$.

``Only If'':\; Since $f\mathcal{D}f'$,
 we get $id_{J_0,J_1}\mathcal{D}id_{J'_0,J'_1}$.
Therefore, by Lemma \ref{idem} (2), $\Sigma_{J_0,J_1}\cong \Sigma_{J'_0,J'_1}$.

``If'':\; Since $\Sigma_{J_0,J_1}\cong \Sigma_{J'_0,J'_1}$, we have $id_{J_0,J_1}\mathcal{D}id_{J'_0,J'_1}$
 by Lemma \ref{idem} (2). Thus, we have $f\mathcal{D}f'$.
\end{proof}

 As a remaining question, it needs to be considered in further work {\em how to characterize the $\mathcal J$-relation in} $\rm End_{par}(\Sigma)$.

As a simple application of  Green's equivalences, we now give an fact of regular elements in $\rm End_{par}(\Sigma(Q))$ for a linear quiver $Q$ of type $A_n$, where $\Sigma=\Sigma(Q)$ means the seed determined in by $Q$, whose cluster variables are correspondent to the vertices of $Q$ including frozen variables corresponding to the frozen vertices of $Q$ and whose exchange matrix is the skew-symmetric matrix determined by the number of arrows between vertices.

Two subsets $S_1, S_2\subseteq$  of $\widetilde X$ are said to {\bf be linked}  if there are $s_1\in S_1, s_2\in S_2$ such that $b_{s_1s_2}\neq 0$.

Each $(I_0,I_1)$-type seed $\Sigma_{I_0,I_1}$ of $\Sigma$ can be decomposed as $\bigsqcup \Sigma_i$, where the mixing-type sub-seed $\Sigma_i$ is determined by a connected component of the corresponding quiver of $\Sigma_{I_0,I_1}$. All such $\Sigma_i$ are called the {\bf connected components} of $\Sigma_{I_0,I_1}$. Write  $\Sigma_i=(X_i,\widetilde {B_i})$. Let $f:\Sigma_{I_0,I_1}\rightarrow\Sigma$ be a partial seed homomorphism, denote $f_i=f\mid_{\Sigma_i}$.

\begin{Lem}\label{reg}
Let $\Sigma=\Sigma(Q)$, with a partial seed homomorphism $f:\Sigma_{I_0,I_1}\rightarrow\Sigma$ in $\rm End_{par}(\Sigma)$. Give
two statements:

(a)\; For $i,j$,  if $f(\widetilde{X_i})$ and $f(\widetilde{X_j})$ are linked, then there is $k\in Q_0$ such that  $f(\widetilde{X_i}\cup\widetilde{X_j})\subseteq f(\widetilde{X_k})$.

(b)\; For $x,y\in \widetilde{X}\setminus I_1$, if $f(x)=f(y)$, then $x,y\in X\setminus (I_0\cup I_1)$ or $x,y\in (X_{fr}\setminus I_1)\cup I_0$.
\\
Then,

(1)\; If $f$ is regular, then (a) and (b) hold.

(2)\; Conversely, assume $Q$ is a linear quiver of type $A_n$, then if (a) and (b) hold, we have  $f:\Sigma_{I_0,I_1}\rightarrow\Sigma$ is regular in $\rm End_{par}(\Sigma)$.
\end{Lem}

\begin{proof}
 Assume that $f(\Sigma_{I_0,I_1})$ is decomposed as $\bigsqcup \Sigma'_{i'}$ (see \cite{HLY}) for a finite index $I'$. Denote $\Sigma'_{i'}=(X'_{i'},\widetilde{B'_{i'}})$, where the mixing-type sub-seed $\Sigma'_{i'}$ is determined by a connected component of the corresponding quiver of $f(\Sigma_{I_0,I_1})$.

(1) We have $g\in \rm End_{par}(\Sigma)$ such that $f\circ g\circ f=f$. Without loss of generality, assume $Dom(g)=f(\Sigma_{I_0,I_1})$. The condition (b) holds. Otherwise, there exist $x\in X\setminus(I_0\cup I_1)$ and $y\in (X_{fr}\setminus I_1)\cup I_0$ with $f(x)=f(y)$, then $x$ and $y$ can not both in the domain of $f\circ g\circ f$ according to the definition. It is a contradiction.

For (a), if $f(\widetilde X_i)$ and $f(\widetilde X_j)$ are linked, then there exists $x\in f(\widetilde X_i), y\in f(\widetilde X_j)$ such that $b_{xy}\neq 0$, that is, $f(\widetilde X_i)$ and $f(\widetilde X_j)$ are in the same connected component $\Sigma'_{k'}$ of $f(\Sigma_{I_0,I_1})$. Thus, $f(\widetilde{X_i}\cup\widetilde{X_j})=f(\widetilde X_i)\cup f(\widetilde X_j)\subseteq \widetilde{X'_{k'}}$.

As $f\circ g\circ f=f$, so $g(\widetilde{X'_{k'}})\subseteq \widetilde{X_k}$ for some $k$. Since $f(x')=x$ and $f\circ g\circ f=f$, we have $x=fgf(x')\in f(g(\widetilde{X'_{k'}}))\subseteq f(\widetilde X_k)$. Therefore, $x\in f(\widetilde X_k)\cap \widetilde{X'_{k'}}$. As $\Sigma'_{k'}$ is a connect component and $f(\Sigma_{k})$ is connected, we have $f(\widetilde X_k)\subseteq \widetilde{X'_{k'}}$. For any $z\in \widetilde{X'_{k'}}$, there exists $w\in \widetilde X\setminus I_1$ such that $f(w)=z$. By $f\circ g\circ f=f$, we have $fgf(w)=fg(z)=f(w)=z$. Moreover, since $z\in \widetilde{X'_{k'}}$ and $g(\widetilde{X'_{k'}})\subseteq \widetilde{X_k}$, we have $z=fg(z)\in f(\widetilde{X_k})$, it means that $\widetilde{X'_{k'}}\subseteq f(\widetilde{X_k})$. Thus, we obtain that $\widetilde{X'_{k'}}=f(\widetilde{X_k})$. Therefore,  $f(\widetilde{X_i}\cup\widetilde{X_j})=f(\widetilde{X_i})\cup f(\widetilde{X_j})\subseteq \widetilde{X'_{k'}}=f(\widetilde{X_k})$ which means the result.

(2) The condition (a) means that there exists an $i\in I$ such that $f(\widetilde{X_i})=\widetilde{X'_{i'}}$ for each $i'\in I'$. Otherwise, we may assume that $\widetilde{X'_{i'}}=f(\widetilde{X_1})\cup f(\widetilde{X_2})$. As $\Sigma_{i'}$ is connected, so $f(\widetilde{X_1})$ and $f(\widetilde{X_2})$ are linked. Hence by (a) there exists $i$ such that $\widetilde{X'_{i'}}=f(\widetilde{X_1}\cup\widetilde{X_2})\subseteq f(\widetilde{X_i})$. As $\Sigma_{i'}$ is a connected component of $f(\Sigma_{I_0,I_1})$, so we have $\widetilde{X'_{i'}}=f(\widetilde{X_i})$.

The condition (b) ensures that $f(x)\in \Sigma'_{i'}$ is frozen if $x\in \Sigma_i$ is frozen. In fact, note that $x$ is in the image seed $f(\Sigma_{I_0,I_1})$, if $f(x)$ is exchangeable in $\Sigma'_{i'}$, then there is an exchangeable variable $y$  such that $f(y)=f(x)$ by the definition of image seed. It contradicts to (b).

Since $Q$ is a linear quiver of type $A_n$, it is easy to see that $f|_{\Sigma_i}: \Sigma_i\rightarrow \Sigma'_{i'}$ is an isomorphism for each $i$. Thus, for the seed homomorphism $f_1:\Sigma_{I_0,I_1}\rightarrow f(\Sigma_{I_0,I_1})$ satisfying $f_1(x)=f(x)$ for all $x\in \widetilde X\setminus I_1$, we have a seed homomorphism $g:f(\Sigma_{I_0,I_1})\rightarrow \Sigma_{I_0,I_1}$ such that $f_1g=id_{f(\Sigma_{I_0,I_1})}$. Therefore,  $f\mathcal R id_{f(\Sigma_{I_0,I_1})}$ by Lemma \ref{Rclass}, and then $f$ is regular by Proposition \ref{regular}.
\end{proof}

By Lemma \ref{reg}, we have the following byproduct.

\begin{Cor}\label{ret}
Let $Q$ be a linear quiver of type $A_n$ and $\Sigma=\Sigma(Q)$. If a partial seed homomorphism $f:\Sigma_{I_0,I_1}\rightarrow\Sigma$ in $\rm End_{par}(\Sigma)$ is regular, then its $f_1:\Sigma_{I_0,I_1}\rightarrow f(\Sigma_{I_0,I_1})$ is a retraction.
\end{Cor}

\begin{proof} By Lemma \ref{reg} (1), the conditions (a) and (b) hold. Then from the proof of Lemma \ref{reg} (2), we have a seed homomorphism $g:f(\Sigma_{I_0,I_1})\rightarrow \Sigma_{I_0,I_1}$ such that $f_1g=id_{f(\Sigma_{I_0,I_1})}$.
\end{proof}

\subsection{ Continuation of Green's equivalences}

 Here, we characterize in Theorem \ref{greens} the Green's equivalences in $\rm{End}_{par}(\mathcal{A}(\Sigma))$ via the corresponding
relations in $\rm{End}_{par}(\Sigma)$.

{\bf Proof of Theorem \ref{greens}.}

\begin{proof} Denote
$f^S(\Sigma_{I_0,I_1})=\Sigma_{J_0,J_1}$ and
$f'^S(\Sigma_{I'_0,I'_1})=\Sigma_{J'_0,J'_1}$.

First, we prove in the case for $\mathcal F=\mathcal R$. ``If'': \;By Lemma \ref{Rclasswithidentity} and Theorem
\ref{greenrelations} (1), we have $f^S\mathcal{R}id_{J_0,J_1}$ and
$f'^S\mathcal{R}id_{J'_0,J'_1}$. According to Lemma
\ref{retraction}, we have
$f\mathcal{R}id_{\mathcal{A}(\Sigma_{J_0,J_1})}$ and
$f'\mathcal{R}id_{\mathcal{A}(\Sigma_{J_0,J_1})}$. Therefore, we
have $f\mathcal{R}f'$.

``Only If'':\; Since $f\mathcal{R}f'$,  there exist $g,g'\in
\rm End_{par}(\mathcal{A}(\Sigma))$ such that $f=f'\circ_{\Lambda}g'$
and $f'=f\circ_{\Lambda}g$. By Lemma \ref{idealmor.}, we have $g=\tau
g_1$ and $g'=\tau'g'_1$ for some surjective rooted cluster morphisms
$g_1, g'_1$ and injective rooted cluster morphisms $\tau, \tau'$.
Thus, $f\circ_{\Lambda}g=fhg_1=f'$ and
$f'\circ_{\Lambda}g'=f'h'g'_1=f$ for some ideal rooted cluster
morphisms $h$ and $h'$. It follows that
\begin{equation}\label{twoused}
fhg_1=f' \;\;\;\;\; \text{ and}\;\;\;\;\;
f'h'g'_1=f.
\end{equation}
 Since $f$ is noncontractible, we have
$f(x)\not\in \mathbb{Z}$ for all $x\in \widetilde{X}\setminus I_1$.
Then, using (\ref{twoused}), $h'g'_1$ is
noncontractible. Similarly, $hg_1$ is noncontractible. Hence, the partial seed homomorphisms
$(hg_1)^S$ and $(h'g'_1)^S$ determined by $hg_1$ and $h'g'_1$, respectively, due to Proposition \ref{induceseed} are in $\rm End_{par}(\Sigma)$. From (\ref{twoused}) , Lemma \ref{gluingfact} and by definition, we have $f^S\circ(h g_1)^S=f'^S$ and
$f'^S\circ(h' g'_1)^S=f^S$. Therefore, $f^S\mathcal{R}f'^S$.

Second, we prove in the case $\mathcal F=\mathcal L$.
 By Lemma \ref{idealmor.},
we have
$$f:\mathcal{A}(\Sigma_{I_0,I_1})\overset{f_1}{\twoheadrightarrow}\mathcal{A}(\Sigma_{J_0,J_1})\overset{\tau}{\hookrightarrow}\mathcal{A}(\Sigma)\;\;\text{and}\;\;f':\mathcal{A}(\Sigma_{I'_0,I'_1})\overset{f'_1}{\twoheadrightarrow}\mathcal{A}(\Sigma_{J'_0,J'_1})\overset{\tau'}{\hookrightarrow}\mathcal{A}(\Sigma),$$

``Only if'': \;By Theorem \ref{greenrelations} (2), we have
$\Sigma_{I_0,I_1}=\Sigma_{I'_0,I'_1}$ and there exists a seed
isomorphism $\phi:\Sigma_{J_0,J_1}\rightarrow \Sigma_{J'_0,J'_1}$
such that $f'^S=\phi f^S$. By Proposition \ref{basiclem}, there is a
rooted cluster isomorphism
$\widetilde\phi:\mathcal{A}(\Sigma_{J_0,J_1})\rightarrow\mathcal{A}(\Sigma_{J'_0,J'_1})$.
Since $f'^S=\phi f^S$,  we have $\widetilde{\phi} f_1(x)=\phi f^S(x)=f'^S(x)=f'_1(x)$ for any $x\in \widetilde{X}\setminus
I_1$, that is, $f'_1|_{\widetilde{X}\setminus I'_1}=\widetilde{\phi}
f_1|_{\widetilde{X}\setminus I'_1}$. Moreover, since $f'_1$ and
$\widetilde{\phi} f_1$ are noncontractible, by Lemma
\ref{equel}, it follows that $f'_1=\phi f_1$.

Let $g=\tau' \phi$ and $g'=\tau\phi^{-1}$. Therefore,
$g\circ_{\Lambda} f=\tau'\phi h^{J_0,J_1}_{J_0,J_1} f_1=\tau'\phi
id_{\mathcal{A}(\Sigma_{J_0,J_1})} f_1=\tau'f'_1=f'.$ Similarly,  $f=g'\circ_{\Lambda} f'$. Thus, we have $f\mathcal{L}f'$.

``If'': \;Since $f\mathcal{L}f'$, there exist $g,g'\in
\rm End_{par}(\mathcal{A}(\Sigma))$ such that $g\circ_{\Lambda}
f=ghf_1=f'$ and $g'\circ_{\Lambda} f'=g'h'f'_1=f$ for  some $h, h'\in \Lambda$. It is clear that
$\Sigma_{I_0,I_1}=\Sigma_{I'_0,I'_1}$.
 Now we show $f'^{S}=(id_{J'_0,J'_1}\varphi'^S) \circ f^{S}$ with a seed isomorphism $\varphi'^S: \Sigma_{J_0,J_1}\cong \Sigma_{J'_0,J'_1}$.

From
$f(\mathcal{A}(\Sigma_{I_0,I_1}))=\mathcal{A}(\Sigma_{J_0,J_1})$ and
$f'_1(\mathcal{A}(\Sigma_{I_0,I_1}))=\mathcal{A}(\Sigma_{J'_0,J'_1})$, we have
$g'h'(\mathcal{A}(\Sigma_{J'_0,J'_1}))=\mathcal{A}(\Sigma_{J_0,J_1})$.
By Lemma \ref{idealmor.}, there exists a surjective
rooted cluster morphism
$\varphi:\mathcal{A}(\Sigma_{J'_0,J'_1})\rightarrow
\mathcal{A}(\Sigma_{J_0,J_1})$ such that $g'h'=\tau \varphi$. Thus,
$\tau\varphi f'_1=g'h'f'_1=f=\tau f_1$.

But, since $\tau$ is injective, we get $\varphi
f'_1=f_1$. Similarly, there exists a
surjective rooted cluster morphism
$\varphi':\mathcal{A}(\Sigma_{J_0,J_1})\rightarrow
\mathcal{A}(\Sigma_{J'_0,J'_1})$ such that $\varphi' f_1=f'_1$.
Hence, $\varphi \varphi' f_1=f_1$ and $\varphi' \varphi f'_1=f'_1$.

Moreover, since $Imf_1=\mathcal{A}(\Sigma_{J_0,J_1})$ and
$Imf'_1=\mathcal{A}(\Sigma_{J'_0,J'_1})$, we obtain $\varphi
\varphi'=id_{\mathcal{A}(\Sigma_{J_0,J_1})}$ and $\varphi'
\varphi=id_{\mathcal{A}(\Sigma_{J'_0,J'_1})}$. Thus, the restricted
seed homomorphisms $\varphi^S$ and $\varphi'^S$ of $\varphi$ and $\varphi'$, respectively, are isomorphisms:
$$\Sigma_{J'_0,J'_1}\overset{\varphi^S}{\cong}\Sigma_{J_0,J_1}\;\;\text{and}\;\;\Sigma_{J_0,J_1}\overset{\varphi'^S}{\cong}\Sigma_{J'_0,J'_1}.$$

Owing to $f'_1=\varphi' f_1$, analyzing on each initial cluster
variables, we have $f'^{S}=(id_{J'_0,J'_1}\varphi'^S)\circ f^{S}$.

In summary, by Theorem \ref{greenrelations} (2), the result holds for the case $\mathcal F=\mathcal L$.

Since $\mathcal{H}=\mathcal{L}\wedge \mathcal{R}$ and
$\mathcal{H}=\mathcal{L}\vee \mathcal{R}$, the cases
$\mathcal{F}=\mathcal{H}$ and $\mathcal{F}=\mathcal{D}$ follow immediately.
\end{proof}

We call the connection given in this theorem and its proof the {\bf continuation of Green's equivalences} from $\rm{End}_{par}(\Sigma)$ to
$\rm{End}_{par}(\mathcal{A}(\Sigma))$

\subsection{  Iso-classes of sub-rooted cluster algebras via regular $\mathcal D$-classes  }

As the embody of the meaning of Green's equivalences in cluster algebras, we now give in Theorem \ref{number} a one-to-one correspondence between the isomorphic
classes of sub-rooted cluster algebras of $\mathcal{A}(\Sigma)$ and the regular $\mathcal{D}$-classes in $\rm{End_{par}}(\Sigma)$.

An isomorphism of two sub-rooted cluster algebras of a
rooted cluster algebra is under the meaning of rooted cluster
isomorphism, replying on the fact from Proposition \ref{basiclem}.

{\bf Proof of Theorem \ref{number}.}

\begin{proof} By the definition of sub-rooted cluster algebra, $\mathcal A(\Sigma')\cong \mathcal A(\Sigma_{I_0,I_1})$ for some $I_0\subseteq X$ and $I_1\subseteq \widetilde{X}$, then $[\mathcal{A}(\Sigma')]=[\mathcal{A}(\Sigma_{I_0,I_1})]$.
By Proposition \ref{basiclem},   $\Sigma'\cong\Sigma_{I_0,I_1}$ in {\bf
Seed}.

 According to Lemma \ref{idem}, $\varphi$ is well-defined. By
 Corollary
\ref{id}, for any regular $\mathcal{D}$-class $D$, there exist
$I_0$ and $I_1$ such that $id_{I_0,I_1}\in D$, which is
equivalent to $D_{id_{I_0,I_1}}=D$. Therefore, $\varphi([\mathcal A(\Sigma_{I_0,I_1})])=D$, that is, $\varphi$ is
surjective.

If $D=D_{id_{I_0,I_1}}=D_{id_{I'_0,I'_1}}$, then
$id_{I_0,I_1}\mathcal{D}id_{I'_0,I'_1}$. By Lemma \ref{idem}, if we have
$\Sigma_{I_0,I_1}\cong\Sigma_{I'_0,I'_1}$, then $[\mathcal A(\Sigma_{I_0,I_1})]=[\mathcal A(\Sigma_{I'_0,I'_1})]$. It means that  $\varphi$ is
 injective.
\end{proof}

By this theorem and Theorem \ref{rooted cluster subalgebra}, as an application, we obtain the following.
\begin{Cor}
For a rooted cluster algebra $\mathcal{A}(\Sigma)$ and its mixing-type sub-rooted cluster algebra $\mathcal{A}(\Sigma')$, the following
hold:

(1)~$\mathcal{A}(\Sigma')$ is a pure cluster subalgebra of
$\mathcal{A}(\Sigma)$ if and only if there exists $I_0\subseteq X$ such that $id_{I_0,\emptyset}$ belongs to the corresponding regular
$\mathcal{D}$-class of $\mathcal{A}(\Sigma')$.

(2)~$\mathcal{A}(\Sigma')$ is a pure sub-cluster algebra of
$\mathcal{A}(\Sigma)$ if and only if there exists $I_1\subseteq \widetilde{X}$ such that
 $id_{\emptyset,I_1}$ belongs to the corresponding regular
$\mathcal{D}$-class of $\mathcal{A}(\Sigma')$.

(3)~$\mathcal{A}(\Sigma')$ is a rooted cluster
subalgebra of $\mathcal{A}(\Sigma)$ if and only if there exist $I_0\subseteq X$ and $I_1\subseteq \widetilde{X}$ such that $id_{I_0,I_1}$ belongs to  the corresponding regular
$\mathcal{D}$-class of $\mathcal{A}(\Sigma')$ and $b_{xy}=0$ for $x\in X\setminus
(I_0\cup I_1)$ and $y\in I_1$ in $\widetilde{B}$.
\end{Cor}

In summary of the above discussion, for a given rooted cluster
algebra $\mathcal{A}(\Sigma)$, we now obtain the following
conclusions:

(1)~ In a regular $\mathcal D$-class $D$, each $\mathcal{R}$-class
$R$ contains just one idempotent in the form $id_{I_0,I_1}$ (by Lemma
\ref{Rclasswithidentity}, Corollary \ref{id} and Lemma
\ref{identity}(2)).

(2)~ In a regular $\mathcal D$-class $D$, each $\mathcal{L}$-class
$L$ contains at most one idempotent in the form $id_{I_0,I_1}$ (by Lemma
\ref{identity}(1)).

(3)~ The number of $\mathcal{H}$-classes containing identity in the form
$id_{I_0,I_1}$ is equal to that of the regular $\mathcal{R}$-classes
and is not larger than that of the regular $\mathcal{L}$-classes (by
$\mathcal H=\mathcal L\wedge\mathcal R$, (1) and (2) and Proposition
\ref{H}).

(4)~ The $\mathcal{H}$-class $H_{I_0,I_1}$ containing $id_{I_0,I_1}$
is isomorphic to the automorphism group of $\Sigma_{I_0,I_1}$ (by
Lemma \ref{idem}(3)).

(5)~ The number of isomorphism classes of sub-rooted cluster
algebras of type $(I_0,I_1)$ in $\mathcal A(\Sigma)$ is equal to
that of regular $\mathcal D$-classes in $\rm{End}_{par}(\Sigma)$ (by
Theorem \ref{number}).

\section{Sub-rooted cluster algebras from Riemannian surfaces}

In this part, we firstly characterize sub-rooted cluster algebras of rooted cluster algebras from oriented Riemanian surfaces via paunched surfaces. Moreover, we
 build the relation between sub-rooted cluster algebras from Riemannian surfaces and the corresponding surfaces by cutting and paunching at arcs, which is called {paunched surfaces} from the original surfaces, defined in the sequel, and then find the connection to regular $\mathcal D$-classes of the semigroup consisting of partial seed homomorphisms.

Cluster algebras arising form surfaces are introduced in \cite{fosth}\cite{fh}. Let $(S,M)$ be a Riemannian surface with $M$ the set of marked points. A tagged arc is an arc in which each end has been tagged in one of two ways, plain or notched, which satisfies some conditions. A tagged triangulation $T$ is a maximal collection of pairwise compatible tagged arcs in $(S,M)$. A lamination $L$ on $(S,M)$ is a finite collection of non-self-intersecting and pairwise non-intersecting curves in $S$, modulo isotopy relative to $M$, subject to some restrictions. For details, refer to \cite{fh}. A multi-lamination $\mathcal L$ is a finite family of laminations.

It is well known \cite{fh}  that the seeds of a cluster algebra $\mathcal {A}(S,M)$ are one-to-one correspondent to the tagged triangulations of $(S,M)$, as well as the exchangeable variables are one-to-one correspondent to the tagged arcs in $(S,M)$ and the frozen variables are one-to-one correspondent to multi-laminations.

In \cite{fh},  $\tau$ is defined as the map from the untagged arcs to the tagged arcs. More precisely, if $\gamma$ does not cut out a once-punctured monogon (i.e. a loop with a puncture point inside), then $\tau(\gamma)$ is $\gamma$ with both ends tagged plain. Otherwise, let $\gamma$ be a loop, based at a marked point $a$, cutting out a punctured monogon with the sole puncture $b$ inside it. Let $\beta$ be the unique arc connecting $a$ and $b$ and compatible with $\gamma$. Then $\tau(\gamma)$ is obtained by tagging $\beta$ plain at $a$ and notched at $b$.

\begin{Thm}(\cite{fosth},Theorem 7.11)
Assume that $(S,M)$ is not a closed surface with one or two punctures. Let $\mathcal A$ be a cluster algebra with $B(\mathcal A)=B(S,M)$. Then the cluster complex of $\mathcal A$ is isomorphic to the tagged arc complex.
\end{Thm}

\begin{Thm}(\cite{fh},Theorem 13.6)\label{un}
For a fixed tagged triangulation $T$, the map $L\rightarrow (b_{\gamma,L}(T))_{\gamma\in T}$ is a bijection between integral unbounded measured laminations and $\mathbb Z^n$.
\end{Thm}

In Theorem \ref{un}, $b_{\gamma,L}(T)=b_{\gamma}(T,L)$ is the {\bf Shear coordinate} of $L$ with respect to the triangulation $T$, which is defined as a sum of contributions from all intersections of curves in $L$ with the arc $\gamma$. For details, see \cite{fh}.

This theorem means that the laminations on a surface $(S,M)$ are one-by-one correspondence to the frozen variables in the cluster algebras associate with $(S,M)$. So, we also use $x$ to denote the Lamination on $(S,M)$ in case $x$ is a frozen variable.

\begin{Def}\label{paunch}
Assume that $T$ is a tagged triangulation of $(S,M)$, $x$ is an exchangeable cluster variables in the seed $\Sigma(T)$ from $T$. Without loss of confusion, we also write $x$ as the corresponding tagged arc in $T$. Let $(S_x,M)$ be the new Riemaninan surface which is constructed from $(S,M)$ through cutting along the arc $x$. More precisely, there are five cases respectively, see Figure 1 corresponding to (1) $x$ is not a loop with the two end points are on the boundary, (2) $x$ is a loop with the endpoint on the boundary, (3) $x$ connects one marked point on the boundary and one puncture, (4) $x$ is not a loop with the end points are two punctures, (5) $x$ is a loop with the endpoint is a puncture. For details, see \cite{MP}, \cite{QZ}. In order to use such surface in the five cases conveniently, we call it the {\bf $x$-paunched surface} from $(S,M)$.
\end{Def}

\begin{figure}[h] \centering
  \includegraphics*[4,6][440,179]{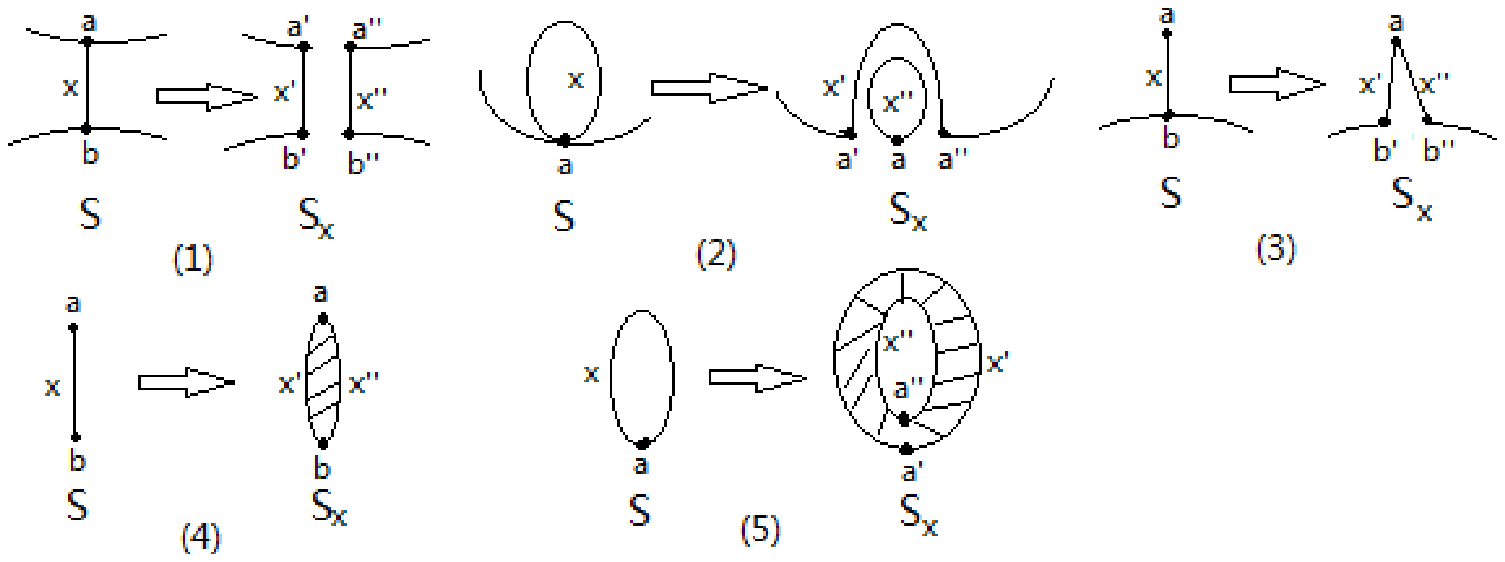}

 Figure 1
\end{figure}

 \begin{figure}[h] \centering
   \includegraphics*[5,16][434,150]{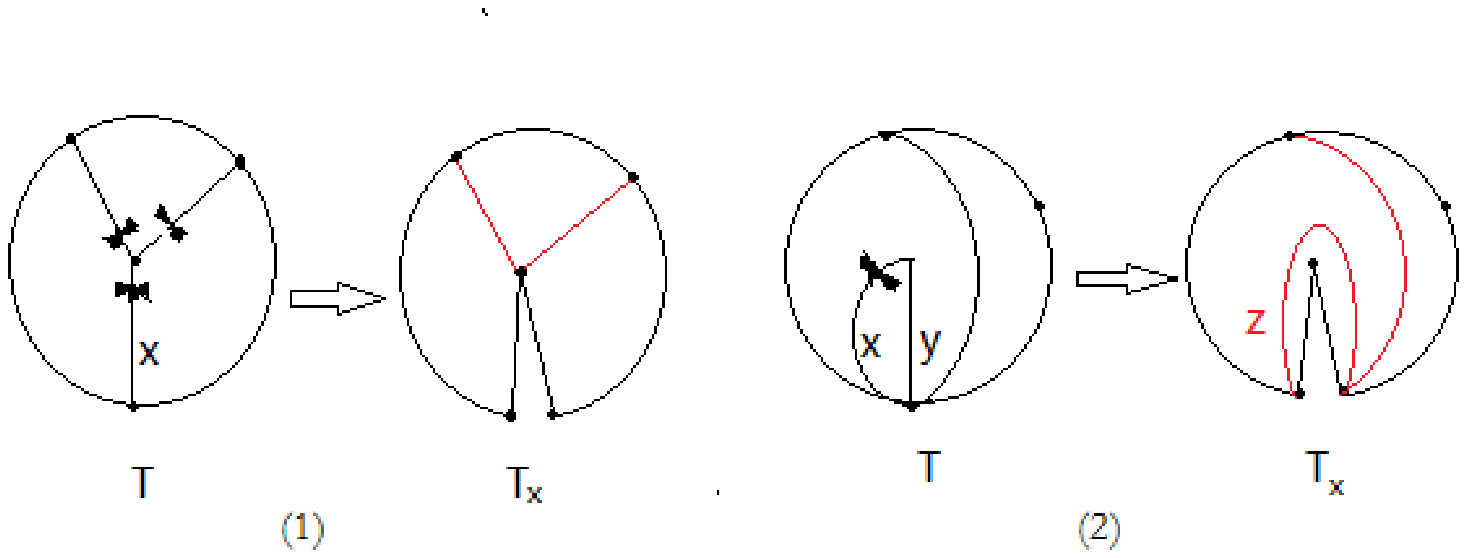}

    Figure 2
\end{figure}

Denote by $T_x$ the corresponding tagged triangulation of $(S_x,M)$. More precisely, assume that $T^{o}$ is the ideal triangulation associated with $T$ and $x^{o}$ is the arc in $T^{o}$ corresponding to $x$. If $x^{o}$ is not in a self-fold triangle, then $T_x$ is obtained by delete $x$ and making all arcs are tagged plain at the end points of $x$. See Figure 2 (1); if $x^{o}$ in a self-fold triangle in $T^{o}$, assume that $y^{o}$ is the another arc of the self-fold triangle and $y$ is the tagged arc in $T$ corresponding to $y^{o}$, then $T_x=(T\setminus \{x,y\})\cup \{z\}$, where $z$ is the unfolded arc in the self-fold triangle, see Figure 2 (2).

Now we give the corresponding lamination $L_x$ when we freeze an exchangeable cluster variable $x$.

In case (1), when $x$ is not a loop with the two end points are on the boundary. Choose an unmarked point $a^+$
($a^{-}$ respectively) on the boundary of $S$ which lying on the counterclockwise (clockwise respectively)
orientation of $a$. Similarly, we can pick up $b^+$and $b^{-}$ on the boundary. Let $L^+_x$ be the curve (up to isotopy relative to $M$) in $S_x$ connecting $a^+$ and $b^{-}$ such that the arcs $\widehat{a'a^+}L^+_x\widehat{b^{-}b'}$ and $x'$ are isotopy relative to $M$, where $\widehat{a'a^+}$ means the arc on the boundary of $S_x$ which from $a'$ to $a^+$. Dually, we can define curve $L^-_x$. In this case, setting $L_x=\{L^+_x,L^-_x\}$. See Figure 3 (1).

In case (2), when $x$ is a loop with the endpoint on the boundary. Choose an unmarked point $a^+$ ($a^{-}$ respectively) on the boundary of $S$ which lying on the clockwise (counterclockwise respectively) orientation of $a'$ ($a''$ respectively). Let $L^+_x$
be the curve (up to isotopy relative to $M$) in $S_x$ connecting $a^+$ and $a^{-}$ such that the arcs $\widehat{a'a^+}L^+_x\widehat{a^{-}a''}$ and $x'$ are isotopy relative to $M$, where $\widehat{a'a^+}$ means the arc on the boundary of $S_x$ which from $a'$ to $a^+$. Let $L^{-}_x$ be a loop (up to isotopy relative to $M$) in the area enclosed by $x''$ such that there are no marked points between $L^{-}_x$ and $x''$ excepting $a$. In this case, setting $L_x=\{L^+_x,L^-_x\}$. See Figure 3 (2).

In case (3), when $x$ connects one marked point on the boundary and one puncture. Choose an unmarked point $a^+$ ($a^{-}$ respectively) on the boundary of $S$ which lying on the counterclockwise (clockwise respectively) orientation of $a'$ ($a''$ respectively). Let $L^+_x$
be the curve (up to isotopy relative to $M$) in $S_x$ connecting $a^+$ and $a^{-}$ such that the arcs $\widehat{a'a^+}L^+_x\widehat{a^{-}a''}$ and $x'$ are isotopy relative to $M$, where $\widehat{a'a^+}$ means the arc on the boundary of $S_x$ which from $a'$ to $a^+$. See Figure 3 (3).

In case (4), when $x$ is not a loop with the end points are two punctures. Let $L_x$ be a loop (up to isotopy relative to $M$) surroundings $x'$ and $x''$ such that there are no marked points between $L_x$ and the loop formed by $x'$ and $x''$ excepting $a$ and $b$. See Figure 3 (4).

In case (5), when $x$ is a loop with the endpoint is a puncture. Let $L^{+}_x$ be a loop (up to isotopy relative to $M$) surroundings $x'$ such that there are no marked points between $L^{+}_x$ and $x^+$ and $x'$ excepting $a'$. Let $L^{-}_x$ be a loop (up to isotopy relative to $M$) in the area enclosed by $x''$ such that there are no marked points between $L^{-}_x$ and $x''$ excepting $a''$. In this case, setting $L_x=\{L^+_x,L^-_x\}$. See Figure 3 (5).

\begin{figure}[h] \centering
  \includegraphics*[3,2][380,178]{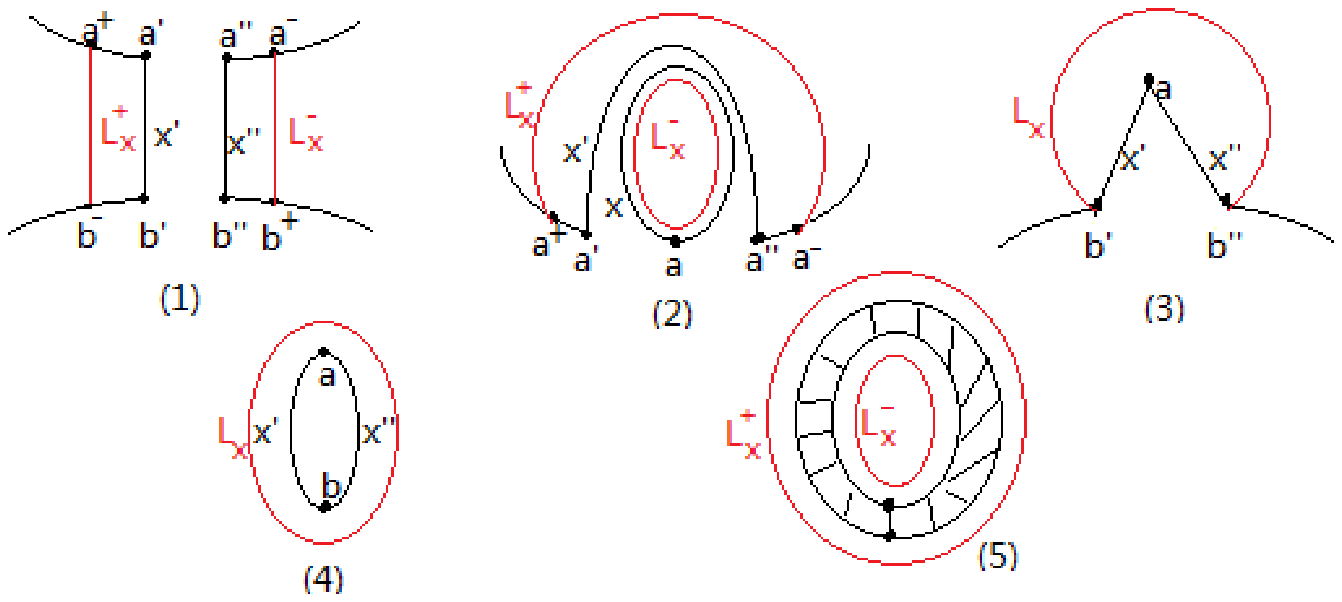}

 Figure 3
\end{figure}

In case $x^{o}$ is not contained in a self-folded triangle and $x$ has just one endpoint tagged notch (we regard a loop $x$ with the endpoints are tagged notch to this case). Let $\gamma$ be a curve in $L$ for some $L\in \mathcal L$. If $\gamma$ does not spiral into the endpoints of $x$, assume that $x$ divides $\gamma$ into $\gamma_1,\cdots,\gamma_k$ for some $k$, define $\gamma^x$ to be subset of $\{\gamma_i\;|\;i=1,\cdots,k\}$ obtained by deleting the curves which are not allowed in a lamination. If $\gamma$ spirals into an endpoint of $x$, let $\gamma'$ be the curve spirals into the same point as $\gamma$ with opposite direction, assume that $x$ divides $\gamma'$ into $\cdots,\gamma'_1,\cdots,\gamma'_k,\cdots$, define $\gamma^x$ to be the subset of $\{\gamma'_i\;|\;i=\cdots,1,\cdots,k,\cdots\}$ obtained by deleting the curves which are not allowed in a lamination. Note that since $\gamma$ spirals into an endpoint of $x$, there are only finite curves in $\gamma^x$. Denote $L^{(x)}=\bigcup\limits_{\gamma\in L}\gamma^x$.

In case $x^{o}$ is not contained in a self-folded triangle and both endpoints of $x$ tagged notch or tagged plain. Let $\gamma$ be a curve in $L$ for some $L\in \mathcal L$. If $\gamma$ does not spiral into the endpoints of $x$, assume that $x$ divides $\gamma$ into $\gamma_1,\cdots,\gamma_k$ for some $k$, define $\gamma^x$ to be subset of $\{\gamma_i\;|\;i=1,\cdots,k\}$ obtained by deleting the curves which are not allowed in a lamination. If $\gamma$ spirals into an endpoint of $x$, assume that $x$ divides $\gamma$ into $\cdots,\gamma_1,\cdots,\gamma_k,\cdots$, define $\gamma^x$ to be the subset of $\{\gamma_i\;|\;i=\cdots,1,\cdots,k,\cdots\}$ obtained by deleting the curves which are not allowed in a lamination. Note that since $\gamma$ spirals into an endpoint of $x$, there are only finite curves in $\gamma^x$. Denote $L^{(x)}=\bigcup\limits_{\gamma\in L}\gamma^x$.

In case $x^{o}$ is contained in a self-folded triangle and $x^{o}$ is folded. Let $\gamma$ be a curve in $L$ for some $L\in \mathcal L$. If $\gamma$ does not spiral into the endpoints of $x$, assume that $z$ divides $\gamma$ into $\gamma_1,\cdots,\gamma_k$ for some $k$, define $\gamma^x$ to be subset of $\{\gamma_i\;|\;i=1,\cdots,k\}$ obtained by deleting the curves which are not allowed in a lamination. If $\gamma$ spirals into an endpoint of $x$, let $\gamma'$ be the curve spirals into the same point as $\gamma$ with opposite direction, assume that $x$ divides $\gamma'$ into $\cdots,\gamma'_1,\cdots,\gamma'_k,\cdots$, define $\gamma^x$ to be the subset of $\{\gamma'_i\;|\;i=\cdots,1,\cdots,k,\cdots\}$ obtained by deleting the curves which are not allowed in a lamination. Note that since $\gamma$ spirals into an endpoint of $x$, there are only finite curves in $\gamma^x$. Denote $L^{(x)}=\bigcup\limits_{\gamma\in L}\gamma^x$.

In case $x^{o}$ is contained in a self-folded triangle and $x^{o}$ is unfolded. Let $\gamma$ be a curve in $L$ for some $L\in \mathcal L$. If $\gamma$ does not spiral into the endpoints of $x$, assume that $x$ divides $\gamma$ into $\gamma_1,\cdots,\gamma_k$ for some $k$, define $\gamma^x$ to be subset of $\{\gamma_i\;|\;i=1,\cdots,k\}$ obtained by deleting the curves which are not allowed in a lamination. If $\gamma$ spirals into an endpoint of $x$, assume that $x$ divides $\gamma$ into $\cdots,\gamma_1,\cdots,\gamma_k,\cdots$, define $\gamma^x$ to be the subset of $\{\gamma_i\;|\;i=\cdots,1,\cdots,k,\cdots\}$ obtained by deleting the curves which are not allowed in a lamination. Note that since $\gamma$ spirals into an endpoint of $x$, there are only finite curves in $\gamma^x$. Denote $L^{(x)}=\bigcup\limits_{\gamma\in L}\gamma^x$.

In both cases, let $\mathcal L^{(x)}=\{L^{(x)}|\;L\in \mathcal L\}$.

Figure 4 illustrates the case that $x^{o}$ is contained in a self-folded triangle and $x^{o}$ is unfolded and $\gamma$ spirals into the endpoints of $x$.

\begin{figure}[h] \centering
  \includegraphics*[2,15][440,180]{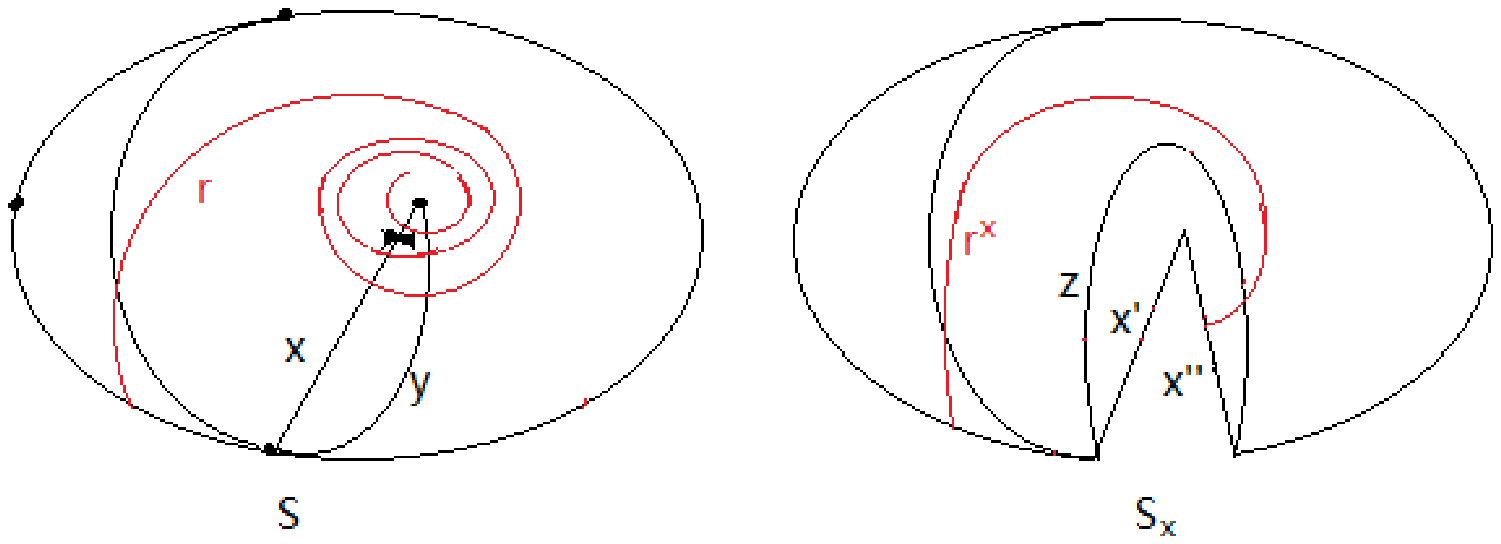}

 Figure 4
\end{figure}

We have the following observations,
\begin{Lem}\label{delete} With the assumptions as above, let $(S,M)$ be a marked surface, $T$ be a tagged triangulation and $\mathcal L$ be a multi-lamination of $(S,M)$. Denote by $\Sigma(S,M,\mathcal L,T)$ the seed determined by $\mathcal L$ and $T$. Let $x$ be a cluster variable of $\Sigma(S,M,\mathcal L,T)$. Then\\
(1)~$\mathcal{A} (\Sigma(S,M,\mathcal L,T)_{\emptyset, \{x\}})$ is a cluster algebra arising from certain surface, more explicitly,

(a)~If $x$ is a frozen variable, then $\mathcal{A} (\Sigma(S,M,\mathcal L,T)_{\emptyset, \{x\}})=\mathcal{A} (\Sigma(S,M,\mathcal L\setminus\{x\}, T))$;

(b)~If $x\in T$ is a tagged arc, then $\mathcal{A} (\Sigma(S,M,\mathcal L,T)_{\emptyset, \{x\}})=\mathcal{A} (\Sigma(S_x,M,\mathcal {L}^x, T_x))$;\\
(2)~$\mathcal{A} (\Sigma(S,M,\mathcal L,T)_{\{x\},\emptyset})$ is a cluster algebra arising from the surface $(S_x,M)$, more explicitly, $$\mathcal{A} (\Sigma(S,M,\mathcal L,T)_{ \{x\},\emptyset})=\mathcal{A} (\Sigma(S_x,M,\mathcal {L}^x\cup\{L_x\},T_x)).$$
\end{Lem}

\begin{proof}
(1)~(a)~ It is clear since the frozen variable $x$ correspondence to the lamination $x$.

(b)~ We can write  $\widetilde{B}(T)=\left(\begin{array}{c}
B(T)\\
(b_{yL})_{y\in T, L\in \mathcal L}
\end{array}\right)$ and $\widetilde{B}(T_x)=\left(\begin{array}{c}
B(T_x)\\
(b'_{yL^{(x)}})_{y\in T_x, L^{(x)}\in \mathcal L^{(x)}}
\end{array}\right)$.

Let $T^{o}$ be the ideal triangulation associate to $T$. By the definition, the exchange matrix $B(T)$ is constructed by $T^{o}$. Let $T_x^{o}$ be the ideal triangulation associate to $T_x$. If $x^{o}$ is not in a self-folded triangle, then $T_x^{o}$ is obtained from $T^{0}$ by deleting $x^{o}$. Thus, for any two arcs $w, w'\neq x^{o}$ in $T^{o}$, we have $b_{ww'}(T^{o})=b_{ww'}(T_x^{o})$. If $x^{o}$ is in a self-folded triangle, assume $z'$ is the folded arc in the self-folded triangle. According to our construction, we have $T^{o}_x=T^{o}\setminus\{z'\}$. Thus, for any two arcs $w, w'\neq z'$ in $T^{o}$, we have $b_{ww'}(T^{o})=b_{ww'}(T_x^{o})$. Therefore, in both cases, we have $B(T)_{\emptyset,\{x\}}=B(T_x)$.

Let $L$ be a lamination in $\mathcal L$, $\gamma\in L$ and $y(\neq x)$ be an arc in $(S,M)$ which does not intersect with $x$.

In case $x, y$ share a puncture $P$ as their common endpoints and $\gamma$ spirals into $P$. Assume that $x$ divides $\gamma$ into $\cdots,\gamma_1,\cdots,\gamma_k,\cdots$. If $\tau^{-1}(x)$ and $\tau^{-1}(y)$ are not in a self-folded triangle, since the arcs in $\{\cdots,\gamma_1,\cdots,\gamma_k,\cdots\}\setminus \gamma^x$ do not contribute to Shear coordinate. Thus, the intersections of $\gamma^{x}$ and $y$ are locally the same as the intersections of $\gamma$ and $y$. If $\tau^{-1}(x)$ and $\tau^{-1}(y)$ are in a self-folded triangle, assume that $z$ is the unfolded arc in the self-folded triangle, since $b_{yL}$ is defined via the arc $z$. By the construction of $\gamma^x$, the intersections of $\gamma^{x}$ and $z$ are locally the same as the intersections of $\gamma$ and $z$.

Otherwise, $x, y$ share a puncture $P$ as their common endpoints but $\gamma$ does not spiral into $P$ or $x, y$ do not share a puncture as their common endpoints. Note that for any curve $\gamma\in L$ which does not spirals into the endpoints of $x$ and arc $y$ which does not share a puncture as there common endpoints, according to our definition of $\gamma^x$, it can been seen that the intersections of $\gamma^{x}$ and $y$ are locally the same as the intersections of $\gamma$ and $y$.

Thus, we have $b_{yL}=b'_{yL^{(x)}}$. It follows $(\widetilde{B}(T))_{\emptyset, x}=\widetilde{B}(T_x)$. Therefore, the result holds.

(2)~
We can write  $\widetilde{B}(T)=\left(\begin{array}{c}
B(T)\\
(b_{yL})_{y\in T, L\in \mathcal L}
\end{array}\right)$ and $\widetilde{B}(T_x)=\left(\begin{array}{c}
B(T_x)\\
(b'_{yL^{(x)}})_{y\in T_x, L^{(x)}\in \mathcal L^{(x)}}\\
(b'_{yL_x})_{y\in T_x}
\end{array}\right)$.

By discussion on the corresponding lamination $L_x$ while freezing an exchangeable variable $x$, analysis case by case (See Figure 3), we can see that the Shear coordinate $b_{y}(T, L_x)$ contributes $+1$ ($-1$ respectively) if and only if there exists a triangle $\Delta$ such that $b_{yx}=1$ ($b_{yx}=-1$ respectively). If the Shear coordinate $b_{y}(T, L_x)$ contributes $0$, then there exists no triangle such that $b_{yx}=1$ or $b_{yx}=-1$. For example, in the case (1), we have illustrated Figure 5 for explanation. In fact, $b_{y}(T, L_x)=b_{yx}=0$, $b_{z}(T, L_x)=b_{zx}=1$ and $b_{z'}(T, L_x)=b_{z'x}=-1$.

Thus, $b_{y x}=b'_{y L_x}$ for all $y\in T$ due to the definition of $b_{yx}$ and $b'_{yL_x}$. Combining the proof of (b), thus, we have $(\widetilde{B}(T))_{x,\emptyset}=\widetilde{B}(T_x)$. Therefore, the result follows.
\end{proof}

\begin{figure}[h] \centering
  \includegraphics*[14,80][107,174]{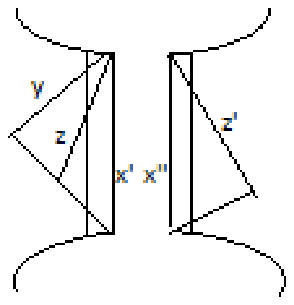}

 Figure 5
\end{figure}

For a set $I$ of some arcs in $T$ (or say, $I\subset X$), denote by $(S_I,M)$ the new Riemaninan surface which is constructed from $(S,M)$ through cutting all arcs $x\in I$, equivalently which is obtained recursively through constructing $x$-paunched surfaces at $|I|$-times. We call $(S_I,M)$ the {\bf $I$-paunched surface} from $(S,M)$.

Denote by $T_I$ the tagged triangulation of $(S_I, M)$ constructed from $T_x$ for all $x\in I$ step by step. For each arc $\gamma$ in $S$, let $I$ divide $\gamma$ into $\cdots,\gamma_1,\cdots,\gamma_k,\cdots$ (An example is given in Figure 6). Let $\gamma^{I}$ be the subset of $\{\cdots,\gamma_1,\cdots,\gamma_k,\cdots\}$ obtained by deleting the arcs which are not allowed in a lamination. If $I$ do not intersect with $\gamma$, then set $\gamma^{I}=\{\gamma$\}. Denote $L^I=\bigcup\limits_{\gamma\in L}\gamma^I$, $\mathcal L^{I}=\{L^I|\;L\in \mathcal L\}$.

\begin{figure}[h] \centering
  \includegraphics*[5,40][376,167]{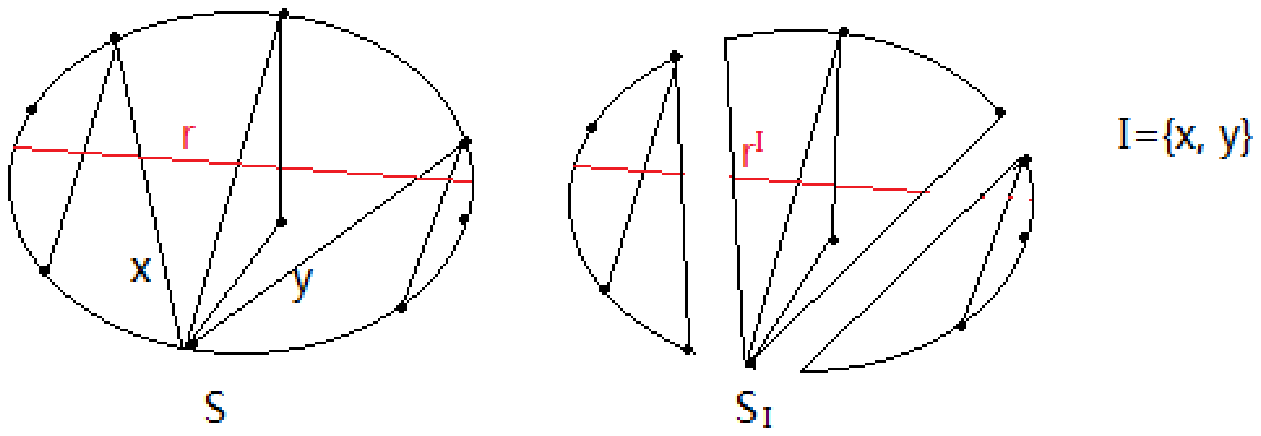}

 Figure 6
\end{figure}

Using the above lemma step by step, we have the following characterization on $\mathcal A(\Sigma(S,M,T,\mathcal L)_{I_0,I_1})$.

\begin{Thm}\label{sur}
 With the assumptions as above, let $(S,M)$ be a marked surface, $T$ be a tagged triangulation and $\mathcal L$ be a multi-lamination of $(S,M)$. Let $I_0$,$I_1$ be two subsets of initial cluster variables with $I_0\cap I_1=\emptyset$, then
$\mathcal{A} (\Sigma(S,M,\mathcal L,T)_{I_0, I_1})$ is a cluster algebra arising from the surface $(S_{(I_0\cup I_1)\cap X},M)$, where $X$ is the cluster in the seed $\Sigma(S,M,\mathcal L,T)$. More explicitly, denote $J=(I_0\cup I_1)\cap X$, then
 $$\mathcal{A} (\Sigma(S,M,\mathcal L,T)_{ I_0,I_1})=\mathcal{A} (\Sigma(S_{J},M,(\mathcal {L}\setminus I_1)^{J}\cup\{L_x\;|\;x\in I_0\}),T_{J}).$$
\end{Thm}

Following this theorem, we calle $(S_{J},M,(\mathcal {L}\setminus I_1)^{J}\cup\{L_x|x\in J\},T_{J})$ the  {\bf $(I_0,I_1)$-paunched surface} from the surface $(S,M,\mathcal L,T)$ for $J=(I_0\cup I_1)\cap X$.

Using of the above results, we now discuss the correspondence between sub-rooted cluster algebras and paunched surfaces.

For two surfaces $(S,M)$ and $(S',M')$, a {\bf homeomorphism} $f$ from $(S,M)$ to $(S',M')$ is a homeomorphism (or say, a topological isomorphism) $f:S\rightarrow S'$ which induces a bijection from $M$ to $M'$. As a homeomorphism, $f$ maps boundary points to boundary points and punctures to punctures. In this case we say $(S,M)$ and $(S',M')$ to be {\bf homeomorphic} under $f$. If $(S,M)$ and $(S',M')$ have respectively tagged triangulations $T$ and $T'$ and   multi-laminations $\mathcal L$ and $\mathcal L'$, we say $(S,M,T,\mathcal L)$ and $(S',M',T',\mathcal L')$ to be {\bf isomorphic} if there is a homeomorphic map $f$ from $(S,M)$ to $(S',M')$ which maps $T$ to $T'$ and induces bijections from $\mathcal L$ to $\mathcal L'$, denote as $(S,M,T,\mathcal L)\cong (S',M',T',\mathcal L')$.

Note that every tagged triangulation is ideal triangulation if the surface without punctures and using Theorem \ref{un}, we have:

\begin{Lem}(Proposition 14.1, \cite{fosth} )\label{isom}
Assume that $(S_1,M_1,\mathcal L_1,T_1)$ and $(S_2,M_2,\mathcal L_2,T_2)$ are two surfaces without punctures and their corresponding quivers $Q(T_1)$ and $Q(T_2)$ are connected.  Assume there is a bijection $\pi: \mathcal L_1\rightarrow \mathcal L_2$ such that $b_{\gamma L}=b'_{\varphi(\gamma)\pi(L)}$ (respectively, $b_{\gamma L}=-b'_{\varphi(\gamma)\pi(L)}$)  for any arc $\gamma$ in $T$ and lamination $L$ in $\mathcal L_1$.
 If $Q(T_1)\overset{\varphi}{\cong} Q(T_2)$ (respectively, $Q(T_1)\overset{\varphi}{\cong} Q(T_2)^{op}$),  then $(S_1,M_1,\mathcal L_1,T_1)\cong(S_2,M_2,\mathcal L_2,T_2)$.
\end{Lem}

This lemma tells us that corresponding quiver is a new algebraic invariant of the oriented surface with tagged triangulation in the case without punctures. However, it is negative in the case with punctures in general.

For example, we have two cluster algebras  $\mathcal A_1=\mathcal A(\Sigma(S_1,M_1,\mathcal L_1,T_1))$ and  $\mathcal A_2=\mathcal A(\Sigma(S_2,M_2,\mathcal L_2,T_2))$ where the surface $(S_1,M_1)$ is the disk with $6$ marked points on the boundary with $\mathcal L_1=\emptyset$ and the surface $(S_2,M_2)$ is the disk with $3$ marked points on the boundary and $1$ puncture with $\mathcal L_2=\emptyset$, see Figure 7. The corresponding quivers, denote $Q_i=Q(T_i)$ for $i=1,2$, are both of type $A_3$. Thus, $\mathcal A_1=\mathcal A(Q_1)\cong \mathcal A_2=\mathcal A(Q_2)$ as rooted cluster algebras.
But, $Q_1$ and $Q_2$ are mutation equivalent. So, there is a triangulation $T_2'$ of $(S_2,M_2)$ such that its corresponding quiver $Q(T_2')$ equals to $Q_1$, i.e. $Q(T_1)=Q(T_2')$. However, $(S_1,M_1,\emptyset,T_1)\not\cong(S_2,M_2,\emptyset,T_2')$ as surfaces with punctures due to $|M_1|\not=|M_2|$.

\begin{figure}[h] \centering
  \includegraphics*[15,42][254,153]{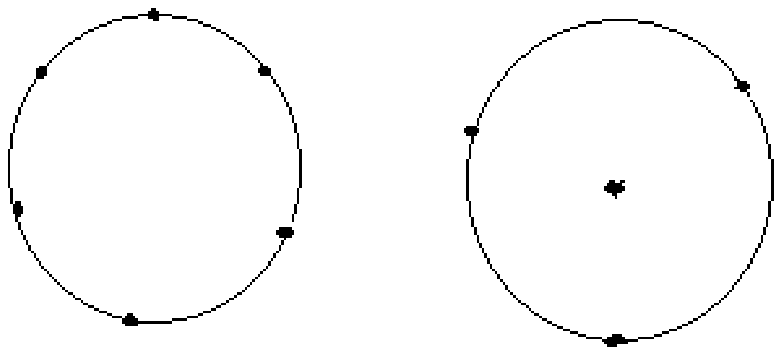}


 Figure 7
\end{figure}

\begin{Rem}
In the case of surfaces with punctures, we refer the readers to \cite{g} for the analogue discussion of Lemma \ref{isom}.
\end{Rem}

From this lemma and combining Theorem \ref{number} and \ref{sur}, we can give the proof of Theorem \ref{final} as follows.

{\bf Proof of Theorem \ref{final}:}

(a)\; It is enough to prove two $(I_0,I_1)$-paunched surfaces of $(S,M,T,\mathcal L)$ are isomorphic if and only if their corresponding rooted $(I_0,I_1)$ sub-cluster algebras of $\mathcal A(\Sigma(S,M,T,\mathcal L))$ are isomorphic.

First,  as $(I_0,I_1)$-paunched surfaces from $(S,M,T,\mathcal L)$, assume $(S_{J},M,T_{J},\mathcal L_1)\overset{\varphi}{\cong} (S_{J'},M,T_{J'},\mathcal L_2),$ where $J=(I_0\cup I_1)\cap X$,  $J'=(I'_0\cup I'_1)\cap X$ and  $\mathcal L_1=(\mathcal {L}\setminus I_1)^{J}\cup\{L_x|x\in J\}$, $\mathcal L_2=(\mathcal {L}\setminus I'_1)^{J'}\cup\{L_x|x\in J'\}$.  Without loss of generality, assume that the  surfaces $S_J$ and $S_{J'}$ are connected respectively,  then the quivers $Q(T_J)$ and $Q(T_{J'})$ are connected respectively, too.

 Under the bijection $\varphi|_{T_1}: T_1\rightarrow T_{2}$, when $\varphi$ preserves the orientation of the surfaces, we have $Q(T_J)\cong Q(T_{J'})$; when $\varphi$ inverses the orientation of the surfaces, we have $Q(T_J)\cong Q(T_{J'})^{op}$. And under the bijection $\varphi|_{\mathcal L_1}: \mathcal L_1\rightarrow  \mathcal {L}_2$, we obtain a one-by-one correspondence from the intersections of laminations in $\mathcal L_1$ with arcs in $T_1$ to the intersections of laminations in $\mathcal L_2$ with arcs in $T_2$. Thus, according to the definition, we have $b_{\gamma L}=b'_{\varphi(\gamma) \varphi(L)}$ if $\varphi$ preserves the orientation, and $b_{\gamma L}=-b'_{\varphi(\gamma) \varphi(L)}$ if $\varphi$ inverses the orientation. In both cases of $\varphi$,  we get $\Sigma(S_{J},M,T_{J},\mathcal L_1)\overset{\varphi}{\cong} \Sigma(S_{J'},M,T_{J'},\mathcal L_2)$ by Definition \ref{seediso}, it follows
$\mathcal A(\Sigma(S,M,T,\mathcal L))_{I_0,I_1}\cong\mathcal A(\Sigma(S,M,T,\mathcal L))_{I'_0,I'_1}$.

Conversely, if $\mathcal A(\Sigma(S,M,T,\mathcal L))_{I_0,I_1}\cong\mathcal A(\Sigma(S,M,T,\mathcal L))_{I'_0,I'_1}$,  then by Proposition \ref{basiclem} and Theorem \ref{sur}, we have
\begin{equation}\label{seediso}
\Sigma(S_{J},M,T_{J},\mathcal L_1)\cong \Sigma(S_{J'},M,T_{J'},\mathcal L_2),
 \end{equation}
 where $J=(I_0\cup I_1)\cap X$, $J'=(I'_0\cup I'_1)\cap X$ and $\mathcal L_1=(\mathcal {L}\setminus I_1)^{J}\cup\{L_x|x\in J\}$, $\mathcal L_2=(\mathcal {L}\setminus I'_1)^{J'}\cup\{L_x|x\in J'\}$.

We may assume that the seeds in (\ref{seediso}) are connected. Then it is obtained that $Q(T_J)\cong Q(T_{J'})$ or $Q(T_J)\cong Q(T_{J'})^{op}$, and the numbers of their frozen variables are equal, which means a bijection between $\mathcal L_1$ and $\mathcal L_2$. Since there is no puncture in $(S,M)$, there are no punctures in $(S_J,M)$ and $(S_{J'},M)$. Hence we see that the conditions of Lemma \ref{isom} are satisfied for the seeds $\Sigma(S_{J},M,T_{J},\mathcal L_1)$ and $\Sigma(S_{J'},M,T_{J'},\mathcal L_2)$. Thus, $(S_{J},M,T_{J},\mathcal L_1)\cong (S_{J'},M,T_{J'},\mathcal L_2)$.

(b)\; It follows immediately from Theorem \ref{number} and (a).  \;\;\;\;\;\;\;\;\;\;\;\;\;\;\;\;\;\;\;\;\;\;\;\;\;\;\;\;\;\;\;\;\;\;\;\;\;\;\;\;\;\;\;\;\;\;\;\;\;\;\;\;\;\;\;\;\;\;\;\;\;\;\;\;\;\;\;$\Box$
\\

With the assumption as in Theorem \ref{final}, but the number of punctures in $(S,M,T,\mathcal L)$ is non-zero, one can discuss the analogue version of Theorem \ref{final}.

{\bf Acknowledgements:}\; This project is supported by the National
Natural Science Foundation of China (No.11271318,  No.11571173) and the
Zhejiang Provincial Natural Science Foundation of China
(No.LZ13A010001).

 The authors warmly thank the referee for
many helpful comments and suggestions in improving the quality and readability of this paper.

\end{document}